\documentclass{amsart}
\usepackage{amssymb,amsmath,array}
\usepackage{pdflscape}
\usepackage{mathrsfs}
\usepackage{tikz}
\usetikzlibrary{matrix, calc, intersections, arrows, shapes}
\usetikzlibrary{arrows.meta, decorations.pathmorphing}
\usepackage{graphicx}
\usepackage{caption}
\usepackage{subcaption}
\usepackage{enumitem}
\usepackage{afterpage}
\usepackage{bm,bbm}
\usepackage{footnote}
\usepackage[hyphens]{url}
\usepackage{hyperref}
\usepackage[hyphenbreaks]{breakurl}
\usepackage{rotating}
\usepackage{blkarray}

\newcommand{\CC}{\mathbb{C}}
\newcommand{\GG}{\mathcal{G}}
\newcommand{\PP}{\mathbb{P}}
\newcommand{\RR}{\mathbb{R}}

\newcommand{\ZZ}{\mathbb{Z}}

\newcommand{\1}{\mathbbm{1}}

\newcommand{\BB}{\mathcal{B}}
\newcommand{\VV}{\mathcal{V}}

\newcommand{\LL}{\Lambda}
\newcommand{\rk}{\textup{rank}\,}

\newcommand{\conv}{\textup{conv}}
\newcommand{\cone}{\textup{cone}}

\newcommand{\GRRmap}{\mathbf{GrR}}
\newcommand{\GVmap}{\mathbf{GV}}
\newcommand{\VGmap}{\mathbf{VG}}
\newcommand{\RGRmap}{\mathbf{RGr}}
\newcommand{\GRVmap}{\mathbf{GrV}}
\newcommand{\VGRmap}{\mathbf{VGr}}

\newcommand{\xx}{\mathbf x}

\renewcommand{\ss}{\mathbf s}

\newcommand{\vv}{\mathbf v}
\newcommand{\qq}{\mathbf q}
\newcommand{\pp}{\mathbf p}
\newcommand{\rr}{\mathbf r}
\newcommand{\ww}{\mathbf w}
\newcommand{\aalpha}{\boldsymbol \alpha}
\newcommand{\mat}{\bm}
\renewcommand{\tt}{\mathbf t}

\newcommand{\bb}{\mathbf b}

\newtheorem{theorem}{Theorem}[section]

\newtheorem{lemma}[theorem]{Lemma}
\newtheorem{corollary}[theorem]{Corollary}
\newtheorem{proposition}[theorem]{Proposition}
\theoremstyle{definition}
\newtheorem{definition}[theorem]{Definition}
\newtheorem{example}[theorem]{Example}

\theoremstyle{remark}
\newtheorem{remark}[theorem]{Remark}

\title[]{Combining realization space models of polytopes}

\author[J. Gouveia]{Jo{\~a}o Gouveia}
\address{CMUC, Department of Mathematics,
  University of Coimbra, 3001-454 Coimbra, Portugal}
\email{jgouveia@mat.uc.pt}

\author[A. Macchia]{Antonio Macchia}
\address{Fachbereich Mathematik und Informatik, Freie Universit\"at Berlin, Arnimallee 2, 14195 Berlin, Germany}
\email{macchia.antonello@gmail.com}

\author[A. Wiebe]{Amy Wiebe}
\address{Fachbereich Mathematik und Informatik, Freie Universit\"at Berlin, Arnimallee 2, 14195 Berlin, Germany}
\email{w.amy.math@gmail.com}

\thanks{Gouveia was supported by the Centre for Mathematics of the University of Coimbra--UID/MAT/00324/2019, funded by the Portuguese Government through FCT/MEC and co-funded by the European Regional Development Fund through the Partnership Agreement PT2020. Macchia was supported by the Einstein Foundation Berlin under Francisco Santos grant EVF-2015-230.}

\begin{document}

\begin{abstract}
In this paper we examine four different models for the realization space of a polytope: the classical model, the Grassmannian model, the Gale transform model, and the slack variety. Respectively, they identify realizations of the polytopes with the matrix whose columns are the coordinates of their vertices, the column space of said matrix, their Gale transforms, and their slack matrices. Each model has been used to study realizations of polytopes. In this paper we establish very explicitly the maps that allow us to move between models, study their precise relationships, and combine the strengths of different viewpoints. As an illustration, we combine the compact nature of the Grassmannian model with the slack variety to obtain a reduced slack model that allows us to perform slack ideal calculations that were previously out of computational reach. These calculations allow us to {answer the question of \cite{CS19}, about the realizability of a family of prismatoids, in general in the negative by proving the non-realizability of one of them.} 
\end{abstract}



\maketitle


%
%


\section{Introduction} \label{sec:introduction}

The study of realization spaces of polytopes has a long history in discrete geometry, dating back at least to the works of Legendre in the 18th century \cite{Leg94}. Given the combinatorics of a polytope, he was interested in knowing how many degrees of freedom its geometric realizations had. In the early 20th century, Steinitz made great advances in the study of realization spaces of $3$-dimensional polytopes \cite{S22,SR76}. From that time onward, the study of realizations of polytopes has stayed a very active and interesting topic in discrete geometry with many important advances \cite{RGZ95}, \cite{RG96}, \cite{Grunbaum}, \cite{D14}, \cite{Z08}, \cite{AZ15}, \cite{F17}. Besides the dimension questions, just deciding if any realizations exist, with or without extra restrictions (such as rationality of the vertex coordinates, or being inscribed in a sphere) is a hard problem which is the subject of much interest.

In order to work with the realization space of a combinatorial polytope, i.e. the set of all its geometric realizations, one has to make it  explicit. These spaces are quite simple for {polytopes of} dimension up to three. The results of Steinitz mentioned above characterize the realizability and the dimension of realization spaces of $3$-dimensional polytopes, and its proof can also be used to show that the realization space of a 3-polytope is {\em contractible} (a trivial set topologically) \cite[Theorem 13.3.3]{RG96}. They can, however become arbitrarily complicated starting from dimension four. This was proven in 1986 by Mn\"ev \cite{Mnev} in the case of rank 3 oriented matroids, and strengthened to include the case of 4-dimensional polytopes by Richter-Gebert in 1996 \cite{RG96}. The Universality Theorem for $4$-polytopes states that for any basic primary semialgebraic set $V$ defined over $\ZZ$ there is a $4$-polytope whose realization space is stably equivalent to $V$.

The more straightforward approach to work with these spaces, which we will call the \textit{classical model}, is to simply take for any realization of a polytope the matrix whose rows are the coordinates of the realizations of the vertices. This is the model Richter-Gebert makes use of, and it has the virtue of being extremely explicit.
One possible alternative is to think of simply the column space of this matrix as a point of the \textit{Grassmannian}. This approach is quite common in the matroid literature for instance \cite[p. 21]{RG96}. Another very classical tool to study the realizations of polytopes is the \textit{Gale diagram}, which replaces the polytope by an arrangement of vectors that codify the affine dependencies of its vertices \cite[Section 5.4]{Grunbaum}. Recently the authors (with Rekha Thomas) introduced another alternative tool to model this space, the \textit{slack variety model}, that represents each polytope by its slack matrix, the matrix attained by evaluating its defining inequalities at all its vertices \cite{GMTWfirstpaper}, \cite{GMTWsecondpaper}.

In this paper we will briefly explain these four ways of looking at the same realization space. Our goal is to highlight how they connect to each other and explicitly and rigorously describe how to go from one model of the realization space to another. The relations between all these sets are schematically presented in Figure~\ref{FIG:relatemodelsmaster}, which includes the maps between sets and references to where in this paper these maps are defined and the relations proved. Our main motivation is to provide tools that allow one to go easily from model to model, combining the strengths of each of them while avoiding their disadvantages.

As an illustration of that philosophy we apply the intuition obtained from the Grassmannian model in order to reduce the size of the slack model, while keeping its expressive power, obtaining a new \textit{reduced slack matrix model} that allows us to make explicit computations in cases where previously the size made them prohibitive.

This paper is a natural continuation of \cite{GMTWfirstpaper} and \cite{GMTWsecondpaper} as it places the slack variety into the larger context of realization space models, and in doing so improves the computational efficiency of the slack model. Large portions of an early version of this paper can also be found in \cite{W19}.

Some of the maps in Figure~\ref{FIG:relatemodelsmaster} are also used to study the slack realization space in the matroid setting \cite{BW19}. In fact, essentially all the results in this paper have a direct translation to the case of matroid realization spaces, which is still an active area of research \cite{RS91}, \cite{FMM13}.

{\bf Organization of the paper.} In Section~\ref{sec:classical} we describe the classical model for realization spaces, as studied in \cite{RG96}, and generalize it to the setting of polyhedral cones. In Section~\ref{sec:grassmannian} we describe the realization space of a $(d+1)$-dimensional polyhedral cone when viewed as a subset of the Grassmannian of $(d+1)$-dimensional subspaces of $\RR^v$. In Theorem~\ref{THM:raysGrassequiv} we prove the equivalence of the Grassmannian and classical models. We then define a convenient algebraic relaxation of this model, called the {\em Grassmannian of the cone}. In Section~\ref{sec:slackmodel} we describe the basic setup of the slack realization space of a cone, as introduced in \cite{GMTWfirstpaper} for polytopes. One of the main results of the paper is found in Theorem ~\ref{THM:slackGrassequiv} which shows that the Grassmannian and slack realization space models are equivalent by providing explicit maps between the two settings. In this section we  see how this result  can be used to provide parameterizations  of the slack variety and a first simplification  of the slack ideal, called the {\em Grassmannian section ideal} of the cone. In Section~\ref{sec:galemodel} we generalize the notion of a Gale transform of a polytope, and give an algebraic characterization of these generalized Gale transforms which lead to a ``dual'' Grassmannian realization space for the cone (Corollary \ref{COR:dualGrassSlackEquiv}). In Section~\ref{sec:applications} we provide an application of the results of the previous sections, namely the equivalence of the slack and other models, to further simplify the slack ideal. We show how the size of a slack matrix can be reduced, improving computational efficiency of the model while retaining all the information of the full slack matrix (see Proposition \ref{PROP:reducedSlackModel}). This is illustrated on several examples, including the well-studied Perles non-rational polytope, introduced in \cite{Grunbaum}, and studied  from a slack perspective in \cite{GMTWsecondpaper}. We also consider {a large quasi-simplicial 4-sphere constructed by Criado and Santos \cite{CS19}, the realizability of which was previously unknown.} 

{\bf Acknowledgments.} We would like to thank Rekha Thomas for  many helpful discussions and her feedback on early versions of this paper. We also want to thank Lauren Williams and  Melissa Sherman-Bennett for pointing out the possible connection between the Grassmannian model and positroids. All the computations of this paper were performed with the aid of the \texttt{SlackIdeals} package \cite{MW20}, available at \url{http://www2.macaulay2.com/Macaulay2/doc/Macaulay2-1.16/share/doc/Macaulay2/SlackIdeals/html/index.html}, developed for the software \textit{Macaulay2} by the second and third author.

\begin{figure}[ht!]
\vspace*{2cm}
\begin{minipage}[c]{\textwidth}
{\renewcommand{\arraystretch}{0.75}
\scalebox{0.8}{
\begin{tikzpicture}[scale=0.8]
\draw[thick] (10,-10) node[draw, fill=blue!25, rounded corners = 10pt, minimum size = 55pt, inner sep =-1pt, label=below:Slack Model] (slack)
{\begin{tabular}{c} $\VV_+(I_K)/\RR^f$ \end{tabular}};
\draw[thick] (10,-13.5) node[draw,  rounded corners = 10pt, minimum size = 35pt, inner sep =-1pt] (Rslack)
{\begin{tabular}{c} $\VV(I_K)^*/\RR^f$ \end{tabular}};
\draw [->,decorate, decoration={snake,amplitude=.4mm,segment length=2mm,post length=1mm}] (slack)++(0,-1.8) -- node[label={[label distance=-5pt]right:{\raisebox{2pt}{\small relax}}}] {} (Rslack);
\draw[thick] (10,0) node[draw, fill=blue!25,  rounded corners = 10pt, inner sep = -2pt, minimum size=55pt, label=above:Grassmannian Model] (grass)
{\begin{tabular}{c}$\LL\in Gr_+(K)$ \\ with $\1\in\LL$ \end{tabular}};
\draw[thick,dashed] (10,3.15) node[draw,   rounded corners = 10pt, inner sep = -2pt, minimum size=25pt] (Kgrass)
{\begin{tabular}{c}$Gr_+(K)$ \end{tabular}};
\draw[thick] (10,4.8) node[draw,   rounded corners = 10pt, inner sep = -2pt, minimum size=35pt] (Rgrass)
{\begin{tabular}{c}$Gr(K)$ \end{tabular}};
\draw [->,decorate, decoration={snake,amplitude=.4mm,segment length=2mm,post length=2mm}] (grass)++(0,1.7) -- node[label={[label distance=-5pt]right:{\small relax}}] {} (Kgrass) -- (Rgrass);
\draw[thick] (0,0) node[draw, fill=blue!25, rounded corners = 10pt, inner sep = 1pt, minimum size=55pt, label=above:Classical Model] (stand)
{\begin{tabular}{c} $\mathcal{R}(P)$ modulo \\ affine transformation \end{tabular}};
\draw[thick,dashed] (0,3.15) node[draw,  rounded corners = 10pt, inner sep = 1pt, minimum size=25pt] (Kstand)
{\begin{tabular}{c}$\mathcal{R}(K)/GL_{d+1}(\RR)$ \end{tabular}};
\draw[thick] (0,4.8) node[draw,  rounded corners = 10pt, inner sep = 1pt, minimum size=35pt] (Rstand)
{\begin{tabular}{c}$\RR^{v\times (d+1)}/GL_{d+1}(\RR)$ \\ with prescribed  ``combinatorics''\end{tabular}};
\draw [->,decorate, decoration={snake,amplitude=.4mm,segment length=2mm,post length=2mm}] (stand)++(0,1.7) --node[label={[label distance=-5pt]left:{\small relax}}] {} (Kstand) -- (Rstand);
\draw[thick] (0,-10) node[draw, fill=blue!25, rounded corners = 10pt, minimum size=55pt, label=below:Gale Model] (gale)
{\begin{tabular}{c} Gale transforms \\ in $Gr^*(K)$ \end{tabular}};
\draw[thick] (0,-13.5) node[draw,  rounded corners = 10pt, minimum size=35pt] (Rgale)
{\begin{tabular}{c} $Gr^*(K)$ \end{tabular}};
\draw [->,decorate, decoration={snake,amplitude=.4mm,segment length=2mm,post length=1mm}] (gale)++(0,-1.8) -- node[label={[label distance=-5pt]left:{\raisebox{2pt}{\small relax}}}] {} (Rgale);
\draw[thick] (6,-13.5) node[draw,  rounded corners = 10pt, minimum size = 35pt, inner sep =-1pt] (section)
{\begin{tabular}{c} $\VV(I_{d+1,v}(K))^*$ \end{tabular}};
\draw[{<[length=3mm,width=2mm,angle'=45,open]}-{>[length=3mm,width=2mm,angle'=45,open]},thick]  (Rslack) -- 
(section);
\draw[-{>[length=4mm,width=2mm,angle'=45,open]},thick]
(stand) edge[bend right=20] node[label={[label distance = -8pt]left:{\begin{tabular}{c} space of \\ affine \\ dependencies\end{tabular}}}]{} (gale);
\draw[{<[length=4mm,width=2mm,angle'=45,open]}-,thick]
(stand) edge[bend right=0] node[label={[label distance = -8pt]right:{\begin{tabular}{c} see \\ \cite[\S5]{Grunbaum}\end{tabular}}}]{} (gale);\draw[{<[length=4mm,width=2mm,angle'=45,open]}-{>[length=4mm,width=2mm,angle'=45,open]},thick]
(gale) edge node[label={[label distance=-8pt]150:$\perp$}]{} (grass);
\draw[-{>[length=4mm,width=2mm,angle'=45,open]},thick]
(grass) edge[bend right=20] node[label={[label distance =-5pt]above:$\GRRmap$}]{}  (stand);
\draw[{<[length=4mm,width=2mm,angle'=45,open]}-,thick]
(grass) edge[bend left=0] node[label={[label distance =-5pt]above:$\RGRmap\,(=\rho)$},swap]{} (stand);
\draw[-{>[length=4mm,width=2mm,angle'=45,open]},thick]
(gale) edge[bend right=0] node[label={[label distance =-2pt]below:$\GVmap = \GRVmap\,\circ\perp$}]{}  (slack);
\draw[{<[length=4mm,width=2mm,angle'=45,open]}-,thick]
(gale) edge[bend right=20] node[label={[label distance =-1pt]below:$\VGmap = \perp\circ\,\VGRmap$},swap]{} (slack);
\draw[-{>[length=4mm,width=2mm,angle'=45,open]},thick]
(grass) edge[bend left=0] node[label={[label distance =-5pt]right:$\GRVmap$}]{}  (slack);
\draw[{<[length=4mm,width=2mm,angle'=45,open]}-,thick]
(grass) edge[bend left=20] node[label={[label distance =-5pt]right:$\VGRmap$},swap]{} (slack);
\draw (5.2,-0) node[label=below:{\begin{tabular}{c}Theorem~\ref{THM:raysGrassequiv} \\[5pt] Corollary~\ref{COR:polytopeGrassequiv} \\[5pt] Corollary~\ref{COR:matrixGrassequiv}\end{tabular}}] {};
\draw (8.6,-4.6) node[label=below:{\begin{tabular}{c}Theorem~\ref{THM:slackGrassequiv} \end{tabular}}] {};
\draw (8.2,-14.1) node[label=below:{\begin{tabular}{c}Corollary~\ref{COR:sectionvariety} \end{tabular}}] {};
\draw (5.2,-9.3) node[label=below:{\begin{tabular}{c}Corollary~\ref{COR:slackDgrassequiv} \end{tabular}}] {};
\end{tikzpicture}
}
}\vspace{-10pt}
\caption{Realization space models for polytope $P$ with $K=P^h$ the homogenization cone of $P$ and the relationships between them.}
\label{FIG:relatemodelsmaster}
\end{minipage}
\vspace*{0.5cm}
\end{figure}


%
%


\section{The classical model} \label{sec:classical}

Let $P$ be an abstract labeled $d$-polytope with $v$ vertices and $f$ facets.
A \textit{realization} of $P$ is an assignment of coordinates to each vertex label, $i \mapsto \qq_i\in\RR^k$ so that the polytope $Q = \conv\{\qq_1,\ldots, \qq_v\}$ is combinatorially equivalent to $P$; that is, it has the same face lattice as $P$. 
\\

\noindent\begin{minipage}{0.61\textwidth}
\begin{example}
If $P$ is a pentagon with $5$ vertices organized into facets $\{1,2\}$, $\{2,3\}$, $\{3,4\}$, $\{4,5\}$, and $\{5,1\}$, then
\[
\begin{array}{ll} \qq_1 = (0,0), &  \qq_2 = (1,0), \\  \qq_3 = (2,1), &  \qq_4 = (1,2), \\  \qq_5 = (0,1) \in \RR^2 \end{array}
\]
gives a realization $Q = \conv\{\qq_1,\ldots,\qq_5\}$ of $P$.
\label{EX:pentagon}
\end{example}
\end{minipage}\hspace{10pt}\begin{minipage}{0.35\textwidth}
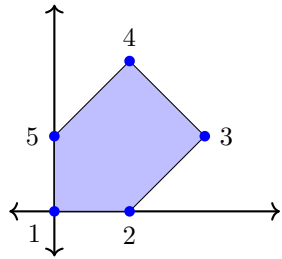
\begin{tikzpicture}
\draw[<->, thick] (0,-0.6) -- (0,2.75);
\draw[<->, thick] (3,0) -- (-0.6,0);
\draw[fill=blue!25] (0,0) -- (1,0) -- (2,1) -- (1,2) -- (0,1) -- (0,0);
\draw (0,0) node[circle, fill, blue, inner sep = 0pt, minimum size = 4pt, label=-135:$1$]{} ;
\draw (1,0) node[circle, fill, blue, inner sep = 0pt, minimum size = 4pt, label=below:$2$]{} ;
\draw (2,1) node[circle, fill, blue, inner sep = 0pt, minimum size = 4pt, label=right:$3$]{} ;
\draw (1,2) node[circle, fill, blue, inner sep = 0pt, minimum size = 4pt, label=above:$4$]{} ;
\draw (0,1) node[circle, fill, blue, inner sep = 0pt, minimum size = 4pt, label=left:$5$]{} ;
\end{tikzpicture}
\end{minipage} \\

The {\em realization space} of (abstract) $P$ is the set of all (concrete) polytopes which are combinatorially equivalent to $P$.

In the classical model for the realization space of $P$ we identify each realization of $P$ with the set of its vertices in $d$-dimensional space, i.e.,
\[
\mathcal{R}(P) := \left\{\mat{Q}=\begin{bmatrix} \qq_1^\top \\ \vdots \\\qq_v^\top\end{bmatrix}\in\RR^{v\times d} : Q=\conv\{\qq_1,\ldots,\qq_v\} \text{ is a realization of } P\right\}.
\]
One often wishes to consider realizations equivalent if they are related via some simple transformation.

Usually (see \cite{RG96}) one would fix an {\em affine basis}, that is, $d+1$ necessarily affinely independent vertices of $P$, whose coordinates would remain fixed in all realizations in $\mathcal{R}(P)$, thus modding out affine transformations. This gives an explicit way to model the set of affine equivalence classes of realizations of $P$. 

This realization space model naturally generalizes to polyhedral cones.
Let $K$ be an abstract labeled $(d+1)$-polyhedral cone with $v$  extreme rays  and $f$ facets. A realization of~$K$ is an assignment of coordinates to each extreme ray, $i\mapsto \rr_i\in\RR^{d+1}$. Recording a cone by the generators of its extreme rays, we get the following realization space model
\[
\mathcal{R}(K) = \left\{\mat{R}=\begin{bmatrix}\rr_1^\top \\ \vdots \\ \rr_v^\top\end{bmatrix} \in\RR^{v\times (d+1)} : K' = \cone\{\rr_1,\ldots, \rr_v\} \text{ is a realization of }K\right\}.
\]

Unlike for polytopes, it is possible for different elements of $\mathcal{R}(K)$ to represent the same realization. In particular
\[
\cone\{\lambda_1\rr_1,\ldots,\lambda_v\rr_v\} =  \cone\{\rr_1,\ldots,\rr_v\}
\]
for any choice of $\lambda_1,\ldots,\lambda_v\in\RR_{>0}.$

If $K$ has the same combinatorial type as polytope $P$, we write $K=P^h$. One way to obtain a realization of $P^h$ is to take $K'=Q^h$, the homogenization cone of a realization $\mat{Q}$ of $P$; that is,
\[
\mat{R} = \begin{bmatrix} \mathbbm{1} & \mat{Q} \end{bmatrix} \in \mathcal{R}(P^h) \quad \mbox{whenever }\;\mat{Q}\in\mathcal{R}(P),
\]
where $\mathbbm{1}$ is the all-ones vector.


%
%


\section{The Grassmannian model} \label{sec:grassmannian}

The relationship between the Grassmannian and the classical model of realization space was already noted by Richter-Gebert \cite[p. 21]{RG96}. Here we formalize the definition of the Grassmannian realization space.

We denote the \textit{Grassmannian} by $Gr(d+1,v)$ which is the set of all ${(d+1)}$-dimensional linear subspaces of a fixed $v$-dimensional vector space. In particular, if we consider subspaces of $\RR^v$, then a point of the Grassmannian can be described as the column space of a $v\times (d+1)$ matrix $\mat{X}$; that is, there is a surjective map
\begin{align*}
\rho:  \RR^{v\times (d+1)} & \to Gr(d+1,v) \\
	\mat{X} & \mapsto \textup{column space}(\mat{X}).
\end{align*}
However, two matrices $\mat{X},\mat{Y}$ can have the same column space, which happens exactly if they differ by an element of $GL_{d+1}(\RR)$,
\begin{equation}
\rho(\mat{X}) = \rho(\mat{Y}) \quad\Leftrightarrow\quad \mat{Y} = \mat{XA}, \mbox{ for some } \mat{A}\in GL_{d+1}(\RR).
\label{EQ:colspaceequal}
\end{equation}

Instead of recording points of the Grassmannian by some (non-unique choice of) matrix of basis elements, we will record each subspace by its {Pl\"ucker vector} which is defined as follows.

For a matrix $\mat{X}\in\RR^{v\times (d+1)}$, denote by $\mat{X}_J$ the submatrix of $\mat{X}$ formed by taking the rows indexed by $J\subset[v] = \{1,2,\dots v\}$.
\begin{definition} For a $(d+1)$-subset $J = \{j_0,\ldots, j_d\}$, with each $j_k\in[v]$, the {\em Pl\"ucker coordinate} of $\mat{X}$ indexed by $J$ is $\det(\mat{X}_J)$.  Notice if~$J$ contains repeated elements, the Pl\"ucker coordinate is necessarily zero, and Pl\"ucker coordinates given by rearranging elements of~$J$ are related by the sign of the appropriate permutation. Hence we record only the Pl\"ucker coordinates indexed by sets $J =\{j_0<j_1<\ldots<j_d\} \in \binom{[v]}{d+1}$ in the {\em Pl\"ucker vector} of $\mat{X}$,
\[
pl(\mat{X}) := \big(\det[\mat{X}_J]\big)_{J\subset\binom{[v]}{d+1}}.
\]
\end{definition}

Now if $\mat{X}\in\RR^{v\times (d+1)}$ is any matrix with $\rho(\mat{X}) = \LL$, then the following map is a well-known embedding of the Grassmannian into projective space
\begin{align*}
pl: Gr(d+1,v) &\to \RR\PP^{\binom{v}{d+1}-1} \\
	\LL &\mapsto pl(\mat{X}).
\end{align*}
We call $pl(\mat{X})$ {\em the Pl\"ucker vector} of $\LL$. Notice that this is a well-defined by \eqref{EQ:colspaceequal}, since $pl(\mat{XA}) = pl(\mat{X})\cdot\det(\mat{A})$ so that matrices related by elements of $GL_{d+1}(\RR)$ map to the same point in projective space under $pl$.

The Grassmannian $Gr(d+1,v)$ is cut out as a subvariety of projective space by the {\em Pl\"ucker ideal}, $I_{d+1,v}$ (see, for instance, \cite[Section 2.2]{MS15}). This is an ideal of the polynomial ring in {\em Pl\"ucker variables}, $I_{d+1,v}\subset \RR[\pp] : = \RR[p_{i_0\cdots i_d}: 1\leq i_0< \cdots < i_d\leq v]$, where we assume the variables are listed in colexicographic\footnote{This order is defined by $i_0\cdots i_d <_{\textup{colex}} j_0\cdots j_d$ if $i_k<j_k$ for the largest index $k$ with $i_k\neq j_k$.} order unless otherwise stated. This ideal consists of all polynomials which vanish on every vector of $(d+1)$-minors coming from an arbitrary  $v\times (d+1)$ matrix.

\begin{example}
The column space of the following matrix $\mat{X}$ is a point in $Gr(3,5)$. One can check that the generators of the Pl\"ucker ideal $I_{3,5}$ vanish on its Pl\"ucker vector.
\[
\mat{X} = \begin{bmatrix} 1 & 0 & 0 \\ 1 & 1 & 0 \\ 1 & 2 & 1 \\ 1 & 1 & 2 \\ 1 & 0 & 1 \end{bmatrix}, \quad
\begin{array}{c}
\begin{array}{rccccccccccc}
\scriptstyle i_0i_1i_2 & & \scriptstyle123 &\scriptstyle124 & \scriptstyle134 & \scriptstyle234 & \scriptstyle125 & \scriptstyle135 & \scriptstyle235 & \scriptstyle145 & \scriptstyle245 & \scriptstyle345 \\
pl(\mat{X}) &\!\!\!\!=\!\!\! &[1:&2:&3:&2:&1:&2:&2:&1:&2:&2]
\end{array}\\[5pt]
I_{3,5} = \left\langle \begin{array}{c} p_{235}p_{145}-p_{135}p_{245}+p_{125}p_{345},\\
p_{234}p_{145}-p_{134}p_{245}+p_{124}p_{345},\\
p_{234}p_{135}-p_{134}p_{235}+p_{123}p_{345},\\
p_{234}p_{125}-p_{124}p_{235}+p_{123}p_{245},\\
p_{134}p_{125}-p_{124}p_{135}+p_{123}p_{145} \end{array} \right\rangle.
\end{array}
\]
\label{EX:pentagonGrass}\end{example}

Notice that the rows of the matrix in the above example generate a cone over a pentagon. In general, a matrix whose $v$ rows generate the cone over a $d$-polytope will have rank $d+1$ and hence define a point in $Gr(d+1,v)$.
Thus if $K$ is a $(d+1)$-dimensional polyhedral cone with $v$ (labeled) extreme rays, we can identify each matrix of generators $\mat{R}\in\mathcal{R}(K)$ with the corresponding point in the Grassmannian. We now consider the structure of this resulting realization space of (linear equivalence classes of) a polyhedral cone inside $Gr(d+1,v)$ by looking at how the combinatorics of $K$ impose conditions on the rays of a realization.

It suffices to consider the extreme ray-facet incidences of $K$.
Notice that by convexity, a set $J$ of $d$ linearly independent extreme rays spans a facet of $K$ if and only if every extreme ray of $K$ is in the same (closed) half-space determined by the hyperplane through rays of $J$. That is, $J$ defines a facet of $K$  if and only if for some normal $\aalpha_J$ to the hyperplane containing rays $J$, we have $\langle \aalpha_J, \rr_j\rangle \geq 0$ for all extreme rays $j\in[v]$.
Furthermore, the extreme rays contained in that facet are exactly those $j$ for which $\langle \aalpha_J, \rr_j\rangle = 0$.
We can use basic linear algebra to determine~$\aalpha_J$ via the following linear functional

\begin{equation}
\langle \aalpha_J, \xx\rangle := \det[\,\rr_{j_1} \,|\, \cdots \,|\, \rr_{j_d} \,|\, \xx\,], \quad \xx\in\RR^{d+1},
\label{EQ:facetnormal}
\end{equation}
which vanishes exactly on the $d$-dimensional subspace spanned by ray generators $\{\rr_j\}_{j\in J}$.

This fact tells us that the above conditions defining facets of $K$ depend only on determinants of sets of $d+1$ extreme rays. Thus we can encode the combinatorial information about $K$ as conditions on the Pl\"ucker vector of a point in the Grassmannian.

\begin{definition} Call a set $J\subset[v]$ a {\em facet extension} if it contains $d$ elements which span a facet of~$K$ and a single element not on that facet. Denote the set of facet extensions of $K$ by $\overline{\mathcal{F}}(K) \subset\binom{[v]}{d+1}.$ 
\end{definition}

Then the elements of the Grassmannian corresponding to realizations of $K$ are
\[
Gr_+(K) := \left\{ \begin{array}{cc}
\parbox[c][][c]{2.7cm}{$\LL\in Gr(d+1,v)$\ :} &  \parbox[c][][c]{5cm}{$pl(\LL)_J = 0$ if rays $J$ lie in a facet,\\ $\Delta_Jpl(\LL)_{J} > 0$ if $J\in\overline{\mathcal{F}}(K)$} \end{array}\right\},
\]
where $\Delta_J = \pm 1$ is a sign which depends on the orientation of simplex $J$.

\begin{example} Consider again the example of the cone $K$ over the pentagon with facets $\{1,2\}$, $\{2,3\}$, $\{3,4\}$, $\{4,5\}$, and $\{5,1\}$. Since every set of $3$ indices necessarily contains $2$ consecutive indices, every $J\subset \binom{[5]}{3}$ is a facet extension. Thus
\[
Gr_+(K) := \left\{ \begin{array}{rl}
\LL\in Gr(d+1,v) : & \Delta_Jpl(\LL)_{J} > 0 \text{ for all $J$} \end{array}\right\}.
\]
Notice that the choice of labeling for our polytope will affect the sign factors $\Delta_J$. To see this, consider the following two pentagons, $Q,Q'$, whose homogenization cones have realizations~$\mat{R}, \mat{R}'$, respectively. For $Q$, we see that simplices $123$ and $134$ are both positively oriented. However, in $Q'$, $123$ is negatively oriented, while $134$ is positively oriented.  \\
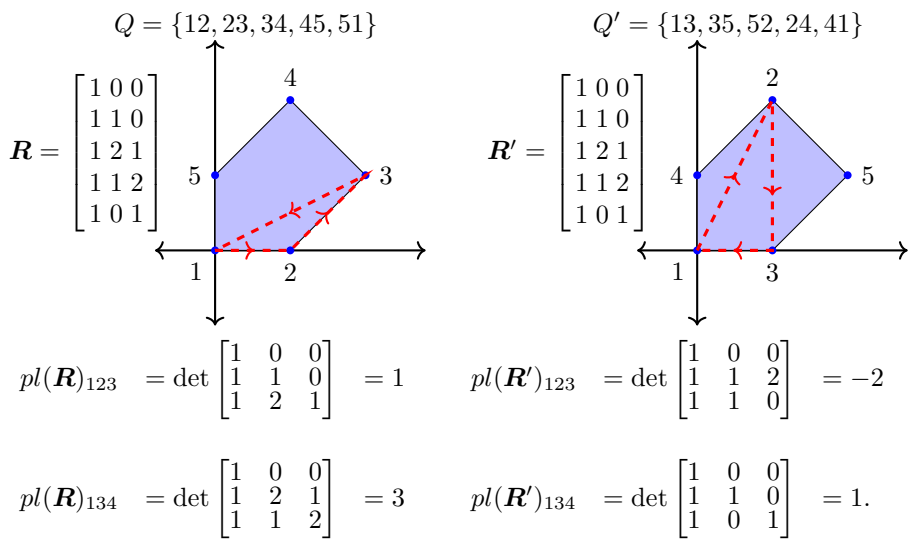
\begin{center}
\hspace{-35pt}\begin{minipage}{0.35\textwidth}
\begin{tikzpicture}
\draw (-1.7,1.3) node[] {$\mat{R} = \left[\begin{array}{@{\,}c@{\;}c@{\;}c@{\,}} 1&0&0 \\ 1&1&0 \\ 1&2&1 \\ 1&1&2 \\ 1&0&1 \end{array}\right]$};
\draw[<->, thick] (0,-1.0) -- (0,2.8);
\draw[<->, thick] (2.8,0) -- (-0.8,0);
\draw[fill=blue!25] (0,0) -- (1,0) -- (2,1) -- (1,2) -- (0,1) -- (0,0);
\draw (0,0) node[circle, fill, blue, inner sep = 0pt, minimum size = 3pt, label=-135:$1$]{} ;
\draw (1,0) node[circle, fill, blue, inner sep = 0pt, minimum size = 3pt, label=below:$2$]{} ;
\draw (2,1) node[circle, fill, blue, inner sep = 0pt, minimum size = 3pt, label=right:$3$]{} ;
\draw (1,2) node[circle, fill, blue, inner sep = 0pt, minimum size = 3pt, label=above:$4$]{} ;
\draw (0,1) node[circle, fill, blue, inner sep = 0pt, minimum size = 3pt, label=left:$5$]{} ;
\draw[red, dashed, very thick] (0,0) -- (1,0) -- (2,1) -- (0,0);
\draw[red, thick, ->] (1,0.5) -- (0.98, 0.49);
\draw[red, thick, ->] (0.5,0) -- (0.52,0);
\draw[red, thick, ->] (1.49, 0.5) -- (1.5, 0.51);
\draw (0.4,3) node[] {$Q = \{12,23,34,45,51\}$};
\end{tikzpicture}
\end{minipage}\hspace{55pt}\begin{minipage}{0.35\textwidth}
\begin{tikzpicture}
\draw (-1.7,1.3) node[] {$\mat{R}' = \left[\begin{array}{@{\,}c@{\;}c@{\;}c@{\,}} 1&0&0 \\ 1&1&0 \\ 1&2&1 \\ 1&1&2 \\ 1&0&1 \end{array}\right]$};
\draw[<->, thick] (0,-1.0) -- (0,2.8);
\draw[<->, thick] (2.8,0) -- (-0.8,0);
\draw[fill=blue!25] (0,0) -- (1,0) -- (2,1) -- (1,2) -- (0,1) -- (0,0);
\draw (0,0) node[circle, fill, blue, inner sep = 0pt, minimum size = 3pt, label=-135:$1$]{} ;
\draw (1,0) node[circle, fill, blue, inner sep = 0pt, minimum size = 3pt, label=below:$3$]{} ;
\draw (2,1) node[circle, fill, blue, inner sep = 0pt, minimum size = 3pt, label=right:$5$]{} ;
\draw (1,2) node[circle, fill, blue, inner sep = 0pt, minimum size = 3pt, label=above:$2$]{} ;
\draw (0,1) node[circle, fill, blue, inner sep = 0pt, minimum size = 3pt, label=left:$4$]{} ;
\draw (0.4,3) node[] {$Q' = \{13,35,52,24,41\}$};
\draw[red, dashed, very thick] (0,0) -- (1,2) -- (1,0) -- (0,0);
\draw[red, thick, ->] (0.5,1) -- (0.51,1.02);
\draw[red, thick, ->] (1,0.75) -- (1,0.73);
\draw[red, thick, ->] (0.5,0) -- (0.48,0);
\end{tikzpicture}
\end{minipage}
\end{center}
\[
{\renewcommand{\arraystretch}{0.75}
\begin{array}{rcl@{\hspace{25pt}}rcl}
pl(\mat{R})_{123} & = \det\begin{bmatrix} 1& 0 & 0 \\ 1 & 1& 0 \\ 1 & 2 & 1 \end{bmatrix} & = 1
& pl(\mat{R}')_{123} & = \det\begin{bmatrix} 1& 0 & 0 \\ 1 & 1& 2 \\ 1 & 1 & 0 \end{bmatrix} & = -2 \\[25pt]
pl(\mat{R})_{134} & = \det\begin{bmatrix} 1& 0 & 0 \\ 1 & 2& 1 \\ 1 & 1 & 2 \end{bmatrix} &= 3
& pl(\mat{R}')_{134} &= \det\begin{bmatrix} 1& 0 & 0 \\ 1 & 1& 0 \\ 1 & 0 & 1 \end{bmatrix} &= 1.
\end{array}}
\]
\label{EX:simplexorient}
\end{example}

\begin{remark}
We note in the above example, that in the labeling of $Q$, the maximal minors of the matrix $\mat{R}$ (whose rows  are generators of the homogenization cone of some realization of $Q$) are all positive. Thus, this $5\times 3$ matrix defines a {\em $(3,5)$-positroid} \cite{P06}. It is possible, as we also see in Example~\ref{EX:simplexorient}, that a different labeling of the same combinatorial type has  matrices $\mat{R}'\in\RR^{n\times k}$ which represent realizations, but which do not correspond to $(k,n)$-positroids. Polytopes that are known to have such labelings include simplices, the square, the pentagon, the triangular prism, the bisimplex, and cyclic polytopes \cite[p. 821]{ARW17}.
\end{remark}

Now we show that $\mathcal{R}(K)$ and $Gr_+(K)$ contain the same information about realizations of $K$ in a very precise sense (see Figure~\ref{FIG:relatemodelsmaster}).

\begin{definition}
Two realizations $K_1,K_2 \subset\RR^{d+1}$ of a $(d+1)$-cone $K$ are called {\em linearly equivalent} if they are related via a linear transformation; that is, $K_2 =\mat{A}K_1$, for some $\mat{A}\in GL_{d+1}(\RR)$.
\label{DEF:coneequiv}
\end{definition}

\begin{theorem}
There is a one-to-one correspondence between points in $Gr_+(K)$ and linear equivalence classes of sets of extreme rays that generate $K$.
\label{THM:raysGrassequiv}
\end{theorem}

\begin{proof} We have already seen that if $\mat{R}\in\mathcal{R}(K)$, then the Pl\"ucker vector of its column space satisfies the Pl\"ucker conditions defining $Gr_+(K)$. Furthermore, if $\mat{R},\mat{Q}\in\mathcal{R}(K)$ are linearly equivalent with $\mat{Q} = \mat{R} \mat{A}$ for some $\mat{A}\in GL_{d+1}(\RR)$, then $pl(\mat{Q}) = \det(\mat{A})\cdot pl(\mat{R})$; in other words, they are the same when considered as points in the Grassmannian via the map $\rho$.

Conversely, given a point  $\LL\in Gr_+(K)$, there exists a matrix $\mat{X}_\LL\in\RR^{v\times (d+1)}$ with that column space. We claim that the map
\begin{eqnarray*}
\GRRmap: Gr_+(K) & \to & \mathcal{R}(K) /GL_{d+1}(\RR) \\
\LL & \mapsto & \mat{X}_\LL
\end{eqnarray*}
is inverse to $\rho$, giving the desired one-to-one correspondence. In fact, it is clear that this is the desired inverse, as long as we show that it is well-defined.  For this, we need to show that the rows of $\mat{X}_\LL$ form a generating set for a cone in the combinatorial class of $K$. To see this,  notice that each facet $F$ of $K$ is spanned by some set of $d$ rays $j_1,\ldots,j_d$. Every other ray $j_0\in[v]$ is either in that facet of $K$, or forms a facet extension with $j_1,\ldots, j_d$. If $j_0\in F$, then by definition of $Gr_+(K)$, rows $j_0,j_1,\ldots, j_d$ of $\mat{X}_\LL$ are linearly dependent. If $j_0\notin F$, then again by definition of $Gr_+(K)$, row $j_0$ of $\mat{X}_\LL$ is in the positive half-space of the hyperplane defined by the rows $j_1,\ldots, j_d$. Thus the rows of $\mat{X}_\LL$ generate a cone having the desired combinatorics.
Finally, any other choice of matrix with $\rho(\mat{Y}) = \LL$ must be of the form $\mat{Y}=\mat{XA}$ for some $\mat{A}\in GL_{d+1}(\RR)$ by \eqref{EQ:colspaceequal}, giving the correspondence up to linear equivalence, as desired.
 \end{proof}

 \begin{remark} Notice that even though rays $\{\rr_1,\ldots, \rr_v\}$ and $\{\lambda_1\rr_1,\ldots,\lambda_v\rr_v\}$ generate the {\em same} cone, they are not necessarily linearly equivalent for $\lambda_1,\ldots, \lambda_v$ not all equal. This is reflected in their Pl\"ucker coordinates, as
\[
pl(\textup{diag}(\lambda_1,\ldots,\lambda_v)\mat{R})_{j_0\cdots j_d} = pl(\mat{R})_{j_0\cdots j_d}\cdot\prod_{i=0}^{d} \lambda_{j_i}.
\]
Thus if there is some $i\neq k$ with $\lambda_i\neq \lambda_k$, we see that $pl(\mat{R})_{ij_1\cdots j_d}$ and $pl(\mat{R})_{kj_1\cdots j_d}$ for $i,k\notin\{j_1,\ldots, j_d\}$ will be scaled differently, resulting in a different point in $Gr_+(K)$ under the map~$\rho$.
\end{remark}

\begin{example} Let $K$ be a cone over a square, with facets $\{12,23,34,14\}$. Then since every set of $3$ rays forms a facet extension, similarly to the pentagon, we get
\[
Gr_+(K) = \{\LL\in Gr(3,4) : pl(\LL) >0\} = \{[w:x:y:z] \in \PP^3 : w,x,y,z > 0\},
\]
where the last equality comes from the fact that there are no Pl\"ucker relations for $d+1=3,v=4$. From the proof of Theorem~\ref{THM:raysGrassequiv}, we know that to obtain a realization of $K$ from a point $[w:x:y:z]\in Gr_+(K)$ we want a matrix $\mat{X}$ such that $pl(\mat{X}) = [w:x:y:z]$.

To mod out by linear transformations, we will fix a choice of $d+1 = 3$ extreme rays in $\mat{X}_\LL$ for every~$\LL$ in advance. Notice that we can fix $d$ arbitrary linearly independent vectors to be a spanning set for some facet, then the remaining vectors are determined by the Pl\"ucker coordinates we wish to impose. So we can choose for example $1\mapsto (0,0,1)^\top$, and $2\mapsto (1,0,1)^\top$. In addition, we can choose a final linear subspace, independent from the first~$d$, for an extreme ray which forms a facet extension with those $d$ indices. In this case, we choose ray $4$ to be the span of vector $(0,1,1)^\top$ and the actual generator, $(0,x,x)^\top$, is determined by the scaling of our Pl\"ucker vector.
Then a representative of the linear equivalence class of generating rays corresponding to each element $[w:x:y:z]$ of $Gr_+(K)$ is, for example, \\
\begin{minipage}{0.5\textwidth}
\[
\mat{X}_\LL = \begin{bmatrix} 0 & 0 & 1 \\ 1 & 0 & 1 \\ y/x & w & (w+y/x)-z/x  \\ 0 & x & x \end{bmatrix}.
\]
\end{minipage}
\begin{minipage}{0.5\textwidth}
\begin{tikzpicture}
\draw[<->] (-2,0) -- (2.6,0);
\draw[<->] (0,-0.8) -- (0,3);
\draw[<->] (-1,-0.5) -- (2,1);
\draw (0,1) node[circle, fill, blue, inner sep = 0pt, minimum size =3pt, label=135:$1$] {};
\draw (-0.7, 0.65) node[circle, fill, blue, inner sep = 0pt, minimum size =3pt, label=200:$2$] {};
\draw (1.5,1.5) node[circle, fill, blue, inner sep = 0pt, minimum size =3pt, label=above:$4$] {};
\draw (0.4,1.2) node[circle, fill, blue, inner sep = 0pt, minimum size =3pt, label=right:$3$] {};
\draw[->, thick, blue] (0,0) -- (0,2.5);
\draw[->, thick, blue] (0,0) -- (-1.4,1.3);
\draw[->, thick, blue] (0,0) -- (2,2);
\draw[->, thick, blue] (0,0) -- (0.8,2.4);
\draw[dashed] (0,1) -- (-0.7,0.65) -- (0.2167,0.65) -- (1,1) -- (0,1);
\end{tikzpicture}
\end{minipage}
\label{EX:coneoversquare}
\end{example}

\begin{remark} Notice that, if $P$ is such that $K=P^h$, Theorem~\ref{THM:raysGrassequiv} does not give us a one-to-one correspondence between points of $Gr_+(K)$ and equivalence classes of realizations of $P$. This is because a linear equivalence class of cone generators need not live in a single affine hyperplane and hence will not correspond to an actual realization of the polytope $P$.
\end{remark}

\begin{example} Continuing Example~\ref{EX:coneoversquare}, we see that the rows of $\mat{X}_\LL$ will give us a realization of a square if and only if they live in some hyperplane $\{\xx\in\RR^3 : \aalpha^\top\xx = \gamma\}$ for some $\aalpha\in\RR^3$ and $\gamma \in \RR$. Solving the equation $\mat{X}_\LL\aalpha = \gamma$, tells us that this only occurs for $[w:x:y:z]$ satisfying $w+y-z=x$. \label{EX:squaregenerators}
\end{example}

It is easy to see that in the above example we may assume $\gamma=1$, so that we only get a square when the all-ones vector $\mathbbm{1}$ is in the column space of $\mat{X}_\LL$. In fact, this is true in general.

\begin{corollary} For $K=P^h$, there is a one-to-one correspondence between points of $Gr_+(K)$ which contain $\mathbbm{1}$ and affine equivalence classes of  realizations in $\mathcal{R}(P)$.
\label{COR:polytopeGrassequiv}
\end{corollary}

\begin{proof} A realization $\mat{R}$ of $K$, gives a realization of $P$ if and only if the rows of~$\mat{R}$ live in a single affine hyperplane of $\RR^{d+1}$; that is, there exists $\aalpha\in\RR^{d+1}$ so that $\mat{R}\aalpha = \mathbbm{1}$. Since the corresponding element of $Gr_+(K)$ is the column space of the matrix $\mat{R}\in\mathcal{R}(K)$, the result follows from Theorem~\ref{THM:raysGrassequiv} by restricting maps~$\rho$ and $\GRRmap$ to matrices with~$\mathbbm{1}$ in the column space and subspaces $\LL$ containing~$\mathbbm{1}$, respectively.
\end{proof}

\subsection{The Grassmannian of a cone}

Notice that the realization space $Gr_+(K)$ defines a semialgebraic subset of the Grassmannian, but can be relaxed to a subvariety, namely the Zariski closure of the following~set.
\begin{definition}
Let the {\em Grassmannian of $K$} be given by relaxing the inequalities in the above realization space to get
\[
Gr(K) := \left\{ \begin{array}{cc}
\parbox[c][][c]{2.8cm}{$\LL\in Gr(d+1,v)$\ :} &  \parbox[c][][c]{3.7cm}{$pl(\LL)_J = 0 \text{ if } J \in \mathcal{F}(K)$\\ $pl(\LL)_{J} \neq 0 \text{ if } J \in \overline{\mathcal{F}}(K)$} \end{array}
\right\},
\]
where we denote by $\mathcal{F}(K)$ the sets $J$ of $d+1$ extreme rays of $K$ such that each~$J$ is contained in some facet of $K$.
\label{DEF:GrassK}
\end{definition}

\begin{example}
Notice that the Grassmannian of the cone over the pentagon only requires all coordinates to be nonzero; it contains no equations and so is full-dimensional in $Gr(3,5)$. Thus its Zariski closure is simply the whole Grassmannian $Gr(3,5)$.
\end{example}

It is not hard to obtain the ideal of this subvariety from the ideal of the whole Grassmannian, the Pl\"ucker ideal $I_{d+1,v}$. From the definition of $Gr(K)$, we have
\[
Gr(K) = \bigg( \VV(I_{d+1,v}) \cap \VV\big(\langle \pp_J : J\in\mathcal{F}(K)\rangle\big)\bigg) \backslash \VV\big(\langle \pp_J : J\in\overline{\mathcal{F}}(K)\rangle\big).
\]
Hence, its ideal (see \cite{CLO15} for necessary correspondences) is given by
\begin{equation}
\mathcal{I}(Gr(K)) = \big(I_{d+1,v} + \langle \pp_J : J\in\mathcal{F}(K)\rangle\big) : \left(\prod_{J\in\overline{\mathcal{F}}(K)}\pp_J\right)^\infty \subset \RR[\pp].
\label{EQ:GrKideal}\end{equation}

\begin{remark}
It is not hard to see from the proof of Theorem~\ref{THM:raysGrassequiv} that as well as specializing to equivalence classes of actual polytope realizations, we can also generalize the map $\GRRmap$ to a map on all of $Gr(K)$, and then $\rho$ generalizes to a map $\RGRmap$ on $v\times(d+1)$ matrices having ``prescribed combinatorics'' up to $GL_{d+1}(\RR)$-action (see Figure~\ref{FIG:relatemodelsmaster}).
\end{remark}

\begin{corollary} There is a one-to-one correspondence between elements of $Gr(K)$ and full rank elements of $\RR^{v\times(d+1)}/GL_{d+1}(\RR)$ whose rows satisfy the\break dependence/independence relations imposed by the combinatorics of $K$. \label{COR:matrixGrassequiv}
\end{corollary}


%
%


\section{The slack model} \label{sec:slackmodel}

The next model represents a polytope or cone by its slack matrix, and it was introduced and studied in \cite{GMTWfirstpaper}, \cite{GMTWsecondpaper}.
The following basic facts about slack matrices of polytopes and cones are from \cite{slackmatrixpaper}, \cite{GMTWfirstpaper}.
If $Q = \conv\{\qq_1,\ldots,\qq_v\}\subset\RR^d$ is a realization of a $d$-polytope with $v$ vertices and $f$ facets whose $\mathcal{H}$-representation is given by $Q = \{\xx\in\RR^d : \mat{W}\xx \leq \ww\}$, then a slack matrix of $Q$ is the nonnegative matrix $S_Q \in \RR^{v \times f}$ whose $(i,j)$-entry is the slack of the $i$th vertex in the $j$th facet inequality: $w_j-\mat{W}_j\qq_i$.

Notice that the choice of $\mathcal{H}$-representation is not unique, and thus a realization does not have a unique slack matrix. More precisely, each realization has an infinite set of slack matrices, every pair of which are related by positive column scalings.

Similarly, if $K$ is a realization of a $(d+1)$-dimensional cone with $v$ extreme ray generators $\rr_1,\ldots, \rr_v$, and $f$ facets with $K = \{\xx\in\RR^{d+1} : \xx^\top \mat{B}\geq \mathbf{0}\}$, then a slack matrix of $K$ is
\[
S_K = \begin{bmatrix}\rr_1^\top \\ \vdots \\ \rr_v^\top \end{bmatrix} \mat{B} \in\RR^{v\times f}.
\]
Notice that if $K =Q^h$ is the homogenization cone of $Q$, that is, $K =\cone\{(1,\qq_1),\ldots,\break (1,\qq_v)\}
= \{(x_0,\xx) \in \RR\times\RR^{d} :  \mat{W}\xx\leq x_0\ww\}$, then $S_K = S_Q$.

Like polytopes, cones do not have unique slack matrices. In particular, recall that the vectors $\gamma_1\rr_1,\ldots, \gamma_v\rr_v$ also generate the cone $K$ for any $\gamma_1,\ldots, \gamma_v\in\RR_{>0}$. More precisely, slack matrices of cones are defined up to positive row and column scaling.

Let us consider two realizations of an abstract polytope. By \cite[Corollary~1.5]{GPRT17}, they are affinely equivalent if and only if they have the same set of slack matrices; they are projectively equivalent if and only if they have the same set of slack matrices up to row scaling. A similar result holds for cones: two realizations of an abstract polyhedral cone are linearly equivalent if and only if they have the same set of slack matrices.

\begin{example}
Recall our first realization of a pentagon from Example~\ref{EX:pentagon},  which was given by $Q = \conv\{ (0,0), (1,0), (2,1), (1,2), (0,1)\}$. Its $\mathcal{H}$-representation is $\{(x,y)^\top\in\RR^2 : y\geq 0, -x+y \geq -1, -x-y \geq -3, x-y \geq -1, x \geq 0\},$
so that
\[
S_Q = \begin{bmatrix} 0 & 1 & 3 & 1 & 0 \\  0 & 0 & 2 & 2 & 1 \\ 1 & 0 & 0 & 2 & 2 \\ 2 & 2 & 0 & 0 & 1 \\ 1 &2 & 2 & 0 & 0
\end{bmatrix}.
\]
It is not hard to check that this is also a slack matrix of the cone~$K$ over this pentagon.
\label{EX:pentagoncone}
\end{example}

Since the affine hull of a realization of $P$ is $d$-dimensional and the linear hull of a realization of $K$ is $(d+1)$-dimensional, we have
\[
\rk\left(\begin{bmatrix} 1 & \qq_1^\top \\ \vdots& \vdots \\ 1 & \qq_v^\top \end{bmatrix}\right) = \rk\left(\begin{bmatrix}\rr_1^\top \\ \vdots \\ \rr_v^\top \end{bmatrix}\right) = d+1,
\]
which implies their slack matrices also have rank $d+1$.
The zeros in each slack matrix also record the extreme point (or ray)-facet incidences of $P$ (or~$K$). Furthermore, we will always have $\mathbbm{1}$ in the column space of the slack matrix of a polytope, whereas $\mathbbm{1}$ is in the column space of $S_{K'}$ if and only if $K'$ is the homogenization cone of a realization of some polytope \cite[Theorem 6]{slackmatrixpaper}.
Interestingly, it follows from \cite[Theorems 22, 24]{slackmatrixpaper} that the above properties are sufficient to characterize slack matrices of polytopes and cones.

To encode these conditions algebraically, we recall the following definitions.

\begin{definition}
The {\em symbolic slack matrix}, $S_P(\xx)$, of an abstract polytope $P$ is obtained by replacing all nonzero entries in a slack matrix $S_P$ with pairwise distinct variables.
The {\em slack ideal} of $P$ is the ideal generated by the $(d+2)$-minors of $S_P(\xx)$, saturated with respect to the product of all variables in $S_P(\xx)$. 
The {\em slack variety} of~$P$ is the complex variety $\mathcal{V}(I_P) \subset \CC^t$, where $t$ is the number of variables in $S_P(\xx)$.
If $\ss \in \CC^t$ is a zero of $I_P$, then we identify it with the matrix $S_P(\ss)$.

The definitions for an abstract polyhedral cone $K$ are the same. In fact, if $K=P^h$, then $S_P(\xx)=S_K(\xx)$ so that their varieties are also the same, $\VV(I_P)=\VV(I_K)$.
\end{definition}

The above facts show us how the slack variety leads to a realization space of a polytope or a cone.

\begin{theorem}[Corollary 3.4 \cite{GMTWfirstpaper}]\
\begin{enumerate}[label = (\roman*)]
\item Given a polytope $P$, there is a bijection between the elements of $\VV_+(I_P)/(\RR^v_{>0}\times\RR^f_{>0})$ and the classes of projectively equivalent polytopes in the combinatorial class of $P$.
\item Given a cone $K$, there is a bijection between elements of $\VV_+(I_K)/(\RR^v_{>0}\times\RR^f_{>0})$ and the classes of linearly equivalent cones in the combinatorial class of $K$. 
\end{enumerate}
\label{THM:slackrealspaces}
\end{theorem}

The space $\VV_+(I_P)/(\RR^v_{>0}\times\RR^f_{>0})$ is called the {\em slack realization space} of $P$ and $\VV_+(I_K)/(\RR^v_{>0}\times\RR^f_{>0})$ is called the {\em slack realization space} of $K$.
Notice that when $K=P^h$, these spaces coincide.

\begin{example} \label{EX:pentagonslack}
Let $P$ be again the pentagon. Then
\[
S_P(\xx) = \begin{bmatrix} 0 & x_1 & x_2 & x_3 & 0 \\ 0 & 0 & x_4 & x_5 & x_6 \\ x_7 & 0 & 0 & x_8 & x_9 \\ x_{10} & x_{11} & 0 & 0 & x_{12} \\ x_{13} & x_{14} & x_{15} & 0 & 0 \end{bmatrix},
\]
which gives a slack ideal $I_P$ with a lexicographic Gr\"obner basis consisting of $25$ polynomials of degree 4.

We construct the slack realization space of the pentagon by first modding out row and column scalings (the action of $\RR^v_{>0}\times\RR^f_{>0}$) by setting a collection of variables in the slack matrix to $1$ (see \cite[Lemma 5.2]{GMTWsecondpaper} for more details on how the entries can be chosen)
\[
S_P(\xx) = \begin{bmatrix} 0 & 1 & x_2 & 1 & 0 \\ 0 & 0 & 1 & 1 & x_6 \\ 1 & 0 & 0 & 1 & 1 \\ 1 & x_{11} & 0 & 0 & x_{12} \\ 1 & x_{14} & x_{15} & 0 & 0 \end{bmatrix}.
\]
Now taking the slack ideal of the above matrix, we find that the slack realization space is the positive part of the variety of
\[
\begin{array}{lll}
\langle x_{11}x_{15}+x_{11}-x_{14}, &x_2x_{14}-x_{14}-x_{15}-1,  & x_{12}x_{14}+x_{11}+x_{12}-x_{14}, \\
\;\; x_6x_{14}+x_6-1, & x_2x_{12}+x_{12}x_{15}-x_{15}, & x_6x_{15}+x_{12}x_{15}+x_{12}-x_{15}, \\
\;\; x_2x_{11}-x_{11}-1, & x_6x_{11}+x_6+x_{12}-1,  & x_2x_6\!-\!x_{12}x_{15}\!-\!x_2\!-\!x_{12}\!+\!x_{15}\!+\!1 \rangle.\end{array}
\]
\end{example}

\subsection{Slack matrices from the Grassmannian}

Recall that the slack matrix of a cone is the product of the matrix whose rows are its extreme rays with the matrix whose columns are its facet normals. Using (\ref{EQ:facetnormal}), 
this means
given a realization $\mat{R}\in\RR^{v\times f}$ with the combinatorics of~$K$, we can calculate the entry of its slack matrix corresponding to ray $i$ and facet $F$ by
\begin{equation} (S_K)_{i,F} =  \det\big[\mat{R}_{J_F}^\top|\rr_i \big] \label{EQ:slackasdets} \end{equation}
where $J_F$ indexes a set of $d$ rays which span facet $F$ of $K$. From this formulation of the slack matrix we can see that the slack matrix of this realization of $K$ can be obtained by filling a $v\times f$ matrix with Pl\"ucker coordinates of $\mat{R}$ (see also \cite{BW19}).

\begin{example}
Let $K$ be a cone over a pentagon with extreme rays generated by the rows of matrix $\mat{X}$ of Example~\ref{EX:pentagonGrass}. The Pl\"ucker vector of $\mat{X}$ is
\[
\begin{array}{rccccccccccc}
\scriptstyle i_0i_1i_2 & & \scriptstyle123 &\scriptstyle124 & \scriptstyle134 & \scriptstyle234 & \scriptstyle125 & \scriptstyle135 & \scriptstyle235 & \scriptstyle145 & \scriptstyle245 & \scriptstyle345 \\[-2pt]
pl(\mat{X}) &= &[1:&2:&3:&2:&1:&2:&2:&1:&2:&2].
\end{array}
\]
The facets of $K$ are spanned by extreme rays $\{1,2\}$, $\{2,3\}$, $\{3,4\}$, $\{4,5\}$ and $\{5,1\}$, and we can calculate the slack matrix of $K$ given in Example~\ref{EX:pentagoncone} from~$\mat{X}$ using \eqref{EQ:slackasdets} as follows:
\[
S_K = \begin{blockarray}{c*{5}{@{\;}c@{\;}}}
 & {12} & {23} & {34} & {45} & {51} \\
 \begin{block}{c@{\;\;}[*{5}{@{\;}c@{\;}}]}
1 & 0 & 1 & 3 & 1 & 0 \\
2 & 0 & 0 & 2 & 2 & 1 \\
3 & 1 & 0 & 0 & 2 & 2 \\
4 & 2 & 2 & 0 & 0 & 1 \\
5 & 1 &2 & 2 & 0 & 0 \\
\end{block}
\end{blockarray} =\!\!\! \begin{blockarray}{c*{5}{@{\;}c@{\;}}}
\\
\begin{block}{c[*{5}{@{\;}c@{\;}}]}
 &  pl(\mat{X})_{121} &  pl(\mat{X})_{231} & pl(\mat{X})_{341} &  pl(\mat{X})_{451} &  pl(\mat{X})_{511} \\
 & pl(\mat{X})_{122}  & pl(\mat{X})_{232} &  pl(\mat{X})_{342} &  pl(\mat{X})_{452}  & pl(\mat{X})_{512} \\
 & pl(\mat{X})_{123} &  pl(\mat{X})_{233} &  pl(\mat{X})_{343} & pl(\mat{X})_{453} &  pl(\mat{X})_{513} \\
 & pl(\mat{X})_{124} &  pl(\mat{X})_{234}  & pl(\mat{X})_{344} & pl(\mat{X})_{454} &  pl(\mat{X})_{514} \\
 & pl(\mat{X})_{125}  & pl(\mat{X})_{235} &  pl(\mat{X})_{345}  & pl(\mat{X})_{455}  & pl(\mat{X})_{515} \\
\end{block}
\end{blockarray}\;.
\]
\label{EX:pentagonPlslack}
\end{example}

Unlike for the pentagon, cones over non-simplicial polytopes will have facets for which there are multiple choices for a spanning set $J$. Denote by $\mathcal{B}_K = \{J_F : F\in\textup{facets}(P)\} \subset\binom{[v]}{d}$ one such set of choices; that is, $J_F$ spans facet $F$ of $K$.
What~\eqref{EQ:slackasdets} tells us is that, given a choice of $\mathcal{B}_K$, we get a map $\GRVmap$ from $Gr_+(K)$ to the slack realization space of $K$ given by
\[
\GRVmap(\LL) :=  \bigg[ \Delta_{i,J}pl(\LL)_{i,J}\bigg]_{i\in[v],J\in\mathcal{B}(K)},
\]
where $\Delta_{i,J}$ is a sign that depends on the orientation of simplex $J$ as well as
the sign of the permutation that orders the elements of $\{i\}\cup J$.

\begin{example} Let $K$ be a cone over a triangular prism with extreme ray generators given by the rows of the following matrix:
\[
\mat{X} = \begin{bmatrix}  1 & 0 & 0 & 0 \\ 1 & 1 & 0 & 0 \\ 1 & 0 & 1 & 0 \\ 1 & 0 & 0 & 1 \\ 1 & 1 & 0 & 1 \\ 1 & 0 & 1 & 2  \end{bmatrix}.
\]
Then the facets of $K$ are indexed by $\{123,456,1245,1346,2356\}$. One choice of $\mathcal{B}_K$ is then $\{123,456,124,136,236\}$ and with this choice we find $\GRVmap(\rho(\mat{X}))$ gives
\begin{small}
\[
S_K \!= \!\! \begin{blockarray}{c*{5}{@{\,}c@{\,}}}
&  \scriptstyle{123} & \scriptstyle{456} &\scriptstyle {124}   &\scriptstyle {136} &\scriptstyle {236} \\
 \begin{block}{c@{\;\;}[*{5}{@{\,}c@{\,}}]}
 1 &  0 &  1 &  0 &  0 & 2 \\
 2 &  0  & 1 & 0  & 2 &  0 \\
 3 & 0 &  2  & 1 &  0 &  0 \\
 4 & 1 &  0 & 0 &  0  & 2 \\
 5 & 1  & 0 & 0  & 2 &  0  \\
 6   & 2  & 0 & 1  & 0 &  0 \\
 \end{block}
 \end{blockarray} = \!\!\!\!\begin{blockarray}{c*{5}{@{\;}c@{\;}}}
 \\
\begin{block}{c[*{5}{@{\;}c@{\;}}]}
 & pl(\mat{X})_{1123} &  pl(\mat{X})_{1456}  &  pl(\mat{X})_{1124} &  pl(\mat{X})_{1136}  & pl(\mat{X})_{1236}\\
 & pl(\mat{X})_{1223}  & pl(\mat{X})_{2456}   & pl(\mat{X})_{1224}  & pl(\mat{X})_{1236}  &  pl(\mat{X})_{2236} \\
 & pl(\mat{X})_{1233} &  pl(\mat{X})_{3456}    &  pl(\mat{X})_{1234} & pl(\mat{X})_{1336} &  pl(\mat{X})_{2336} \\
 & pl(\mat{X})_{1234} &  pl(\mat{X})_{4456} &  pl(\mat{X})_{1244}  & -pl(\mat{X})_{1346}  & pl(\mat{X})_{2346} \\
 & pl(\mat{X})_{1235}  & pl(\mat{X})_{4556}  & -pl(\mat{X})_{1245}  & -pl(\mat{X})_{1356}  &  pl(\mat{X})_{2356}   \\
 & pl(\mat{X})_{1236}  & pl(\mat{X})_{4566}  & -pl(\mat{X})_{1246}   & pl(\mat{X})_{1366}  &  pl(\mat{X})_{2366} \\
\end{block}
\end{blockarray}\,.
\]
\end{small}
Notice that, we could just have easily chosen $\mathcal{B}_K = \{123,456,125,134,356\}$.
\label{EX:triprismslack}
\end{example}

Relaxing this map from the realization space of $K$ to $Gr(K)$, we get
\begin{equation}
\begin{array}{rl}
\GRVmap : Gr(K) \!\!\!& \to \VV(I_K)^* := \VV(I_K)\cap(\RR^*)^t \\[5pt]
\LL \!\!\! & \mapsto \bigg[ \Delta_{i,J}pl(\LL)_{i,J}\bigg]_{i\in[v],J\in\mathcal{B}(K)}
\end{array}
\label{EQ:phidef}
\end{equation}
Notice that this relaxation is well-defined, since $\GRVmap(\LL)$ has the correct zero pattern by definition of $Gr(K)$: any nonzero entry of a slack matrix is indexed by
a facet extension and each zero entry is indexed by vertices contained in a facet of~$K$.
Furthermore, $\GRVmap(\LL)$ has rank $d+1$ since it is obtained by applying linear functionals to some rank $d+1$ matrix whose column space is $\LL$.

Since we  think about this map as filling a slack matrix with Pl\"ucker coordinates, we will use the additional notation $S_K(\LL_{\mathcal{B}})$ to denote the image of $\LL$ under $\GRVmap$ in order to emphasize that the choice of $\mathcal{B}$ determines which Pl\"ucker coordinates are used as slack entries.
We will use the map $\GRVmap$ to prove the following theorem.

\begin{theorem} The nonzero real part of the slack variety of $K$, $\VV(I_K)^*$, up to column scaling, is birationally equivalent to the Grassmannian of the cone $K$, $Gr(K)$. When $K=P^h$ for polytope $P$, $Gr(K)$ is also equivalent to the nonzero real part of the slack variety of $P$, $\VV(I_P)^*$, up to column scaling.
\label{THM:slackGrassequiv} \end{theorem}

To prove this theorem, we define the following reverse map.
\begin{equation}
\begin{array}{rl}
\VGRmap : \VV(I_K)^* & \to Gr(K) \\
\ss &\mapsto \rho(\ss).
\end{array}
\label{EQ:psidef}
\end{equation}
That is, we map a matrix in the slack variety to its column space. This is a $(d+1)$-dimensional subspace of $\RR^v$ by \cite[Theorem 3.2]{GMTWfirstpaper} and it will have the correct Pl\"ucker coordinates because the rows of $S_K(\ss)$ form a realization of $K$ by the reasoning of \cite[Theorem~14]{slackmatrixpaper}.

\begin{proposition} Two elements of $\VV(I_K)^*$ are the same up to column scaling if and only if they have the same column space.
\label{PROP:colscale}\end{proposition}

\begin{proof} Let $\ss,\tt\in \VV(I_K)^*$. Clearly if $S_K(\ss) = S_K(\tt)\cdot D_f$ for some invertible diagonal matrix~$D_f$, then
\[
S_K(\ss) \cdot \RR^f = S_K(\tt)\cdot D_f\cdot \RR^f = S_K(\tt)\cdot \RR^f,
\]
so their column spaces are the same.

Conversely, suppose $S_K(\ss), S_K(\tt)$ have the same column space. For each facet~$F_j$, corresponding to column $j$ of each slack matrix, there exists a flag (maximal chain) through $F_j$ in the face lattice of $K$ from which we obtain a $(d+1)\times(d+1)$ lower triangular submatrix, $S_K(\tt)$, of $S_K(\ss)$ with nonzero diagonal (see \cite[Lemma 3.1]{GMTWfirstpaper}). The columns corresponding to each submatrix will span their respective column spaces, and since these spaces are the same, there must be a change of basis matrix $\mat{A}\in GL_{d+1}(\RR)$ taking one to the other:
{\renewcommand{\arraystretch}{0.8}
\[
\underbrace{\begin{blockarray}{cccc}
&&& F_j  \\
\begin{block}{[cccc]}
\\[-5pt]
\lambda_1 & 0 & \cdots & 0 \\ * & \lambda_2 & \cdots & 0 \\ \vdots & &\ddots & \\ * & & \cdots & \lambda_{d+1} \\ * & & \cdots & * \\ \vdots & && \vdots \\ * & & \cdots & *  \\
\end{block}\end{blockarray}}_{S_K(\ss)} \cdot\mat{A} =
\underbrace{\begin{blockarray}{cccc}
& & &F_j  \\
\begin{block}{[cccc]}
\\[-5pt]
\gamma_1 & 0 & \cdots & 0 \\ * & \gamma_2 & \cdots & 0 \\ \vdots & &\ddots & \\ * & & \cdots & \gamma_{d+1} \\ * & & \cdots & * \\ \vdots & && \vdots \\ * & & \cdots & * \\
\end{block}\end{blockarray}}_{S_K(\tt)}.
\]}

However, from the zero pattern, it is clear that the column corresponding to $F_j$ in the $S_K(\ss)$ can only be a scaling of the same column in $S_K(\tt)$. Since this is true for each facet of $K$, the result follows.
\end{proof}

Since it is clear that $\GRVmap$ and $\VGRmap$ are rational when we record elements of $Gr(K)$ by their Pl\"ucker coordinates, Theorem~\ref{THM:slackGrassequiv} becomes a corollary of the following result.

\begin{lemma} The maps $\GRVmap,\VGRmap$ defined in \eqref{EQ:phidef},\eqref{EQ:psidef}, respectively, are inverses; that is
$\VGRmap\circ\GRVmap = id_{Gr(K)}$ and $\GRVmap\circ\VGRmap = id_{\VV(I_K)^*/\RR^f}$. \label{LEM:inverses} \end{lemma}

\begin{proof}
First consider
\begin{align*}
\VGRmap\circ\GRVmap (\LL) & = \rho(S_K(\LL_{\BB})). \\
\end{align*}
We wish to show that the column space of $S_K(\LL_{\BB})$ is $\LL$.
Let $\mat{X}\in\RR^{v\times(d+1)}$ be a matrix whose columns form a basis for the subspace $\LL$ and let $\aalpha_j$ be the normal to facet $j$ calculated from the rows of $\mat{X}$ as in \eqref{EQ:facetnormal}, so that
\[
S_K(\LL_{\BB}) = \mat{X} \begin{bmatrix} \rule{0.5pt}{7pt} &  & \rule{0.5pt}{7pt} \\ \aalpha_{1} & \cdots & \aalpha_{f} \\ \rule{0.5pt}{7pt} &  & \rule{0.5pt}{7pt} \end{bmatrix}.
\]
Since $\rk(S_K(\LL_{\BB})) = \rk(\mat{X}) = d+1$, we also have $\rk([\aalpha_1\cdots \aalpha_f]) = d+1$, so that in particular, the column spaces of $S_K(\LL_{\BB})$ and $\mat{X}$ are the same. Since the column space of $\mat{X}$ is $\LL$ by definition, this gives
\[
\VGRmap\circ\GRVmap (\LL) = \LL,
\]
as desired.

Next consider
\begin{align*}
\GRVmap\circ\VGRmap (\ss) & =  S_K\big(\rho(\ss)_{\BB}\big).
\end{align*}

By Proposition~\ref{PROP:colscale} it suffices to show that the column space of $S_K(\rho(\ss)_{\BB})$ is the same as the column space of the slack matrix $S_K(\ss)$, but this follows from what we just showed, namely that $\rho(S_K(\LL_{\BB})) = \LL$.
\end{proof}

\begin{remark} Notice that, as in Theorem~\ref{THM:raysGrassequiv}, there may be points in $Gr(K)$ whose image under $\GRVmap$ is not an affine equivalence class of realizations of~$P$, but is simply the orbit of a point in the slack variety of $P$ under column scaling; that is, $S_K(\LL_{\BB})$ might not have $\mathbbm{1}$ in its column space. In fact, we see from the above proof that,  as in Corollary~\ref{COR:polytopeGrassequiv}, $\mathbbm{1}\in\rho(S_K(\LL_{\BB}))$ if and only if $\mathbbm{1}\in\LL$ (see Figure~\ref{FIG:relatemodelsmaster}).
\end{remark}

\begin{example} Recall the cone $K$ over the square of Example~\ref{EX:coneoversquare}. Letting $\LL = [w:x:y:z]\in Gr(K)$, we had
\[
{\renewcommand{\arraystretch}{0.8}
 \mat{X}_\LL = \begin{bmatrix} 0 & 0 & 1 \\ 1 & 0 & 1 \\ y/x & w & (w+y/x)-z/x  \\ 0 & x & x \end{bmatrix}.}
\]
For a square, as for a pentagon, there is only one choice for facet bases $\BB_K$, so that map $\GRVmap$ is unique, and gives
\[
{\renewcommand{\arraystretch}{0.8}
\GRVmap(\LL) = S_K(\LL_{\BB}) = \begin{bmatrix} 0 & w & y & 0 \\ 0 & 0 & z & x \\ w & 0 & 0 & y \\ x & z & 0 & 0\end{bmatrix}.}
\]

Even if we restrict to a point of $Gr_+(K)$, say $w=x=y=1$, $z=2$, it is not hard to check that neither~$\mat{X}$, nor the resulting slack matrix, has $\mathbbm{1}$ in its column space. In fact, we can see that the condition of Example~\ref{EX:squaregenerators} which guarantees the rows of $\mat{X}_\LL$ give the realization of a square, namely that $w+y-z=x$, also guarantees condition $(3)$ of \cite[Theorem 2.2]{GMTWfirstpaper} on the resulting slack matrix.
\label{EX:squareparameters}
\end{example}
Example~\ref{EX:squareparameters} also illustrates another important use of the maps which give us the equivalence of Theorem~\ref{THM:slackGrassequiv}, namely, they allow us to obtain useful parametrizations of the slack variety. By filling a slack matrix with Pl\"ucker coordinates, we are essentially imposing a collection of linear equalities on the slack variety determined by which entries are filled with the same Pl\"ucker coordinates. In some cases, we will see that restricting to such a parametrization allows us to greatly simplify the slack ideal.

If we fix a choice of $\BB_K$ for the cone $K$, then we can consider the ``symbolic image'' of the map $\GRVmap$ as being a matrix in the Pl\"ucker variables $\pp$. Namely, from \eqref{EQ:phidef}, $\GRVmap$ evaluates the following symbolic matrix on the Pl\"ucker vector of $\LL$:
\[
\GRVmap(\pp):=\bigg[ \Delta_{i,J}\pp_{i,J}\bigg]_{i\in[v],J\in\mathcal{B}(K)}.
\]
Notice that the entries of $\GRVmap(\pp)$ need not use all the Pl\"ucker coordinates. In particular, entries are indexed by facets and facet extensions, so that we only require Pl\"ucker variables indexed by $\mathcal{F}(K),\overline{\mathcal{F}}(K)\subset\binom{[v]}{d+1}$, and it is often the case that
\[
\mathcal{F}(K) \cup \overline{\mathcal{F}}(K)\neq \binom{[v]}{d+1}.
\]
Furthermore, by definition of $Gr(K)$, only the variables $\pp_{J}$ for $J\in \overline{\mathcal{F}}(K)$ will be nonzero. In light of this and the fact that the ideal of the Grassmannian of $K$ given in \eqref{EQ:GrKideal} vanishes on the Pl\"ucker coordinates of each $\LL$, and hence on the entries of $\GRVmap(\LL)$, we define the following ideal, which gives the conditions on only the nonzero Pl\"ucker coordinates that will be used in $\GRVmap(\LL)$.

\begin{definition} The {\em Grassmannian section ideal} of $K$, denoted $I_{d+1,v}(K)$ is given by eliminating the variables that are not necessary for $\GRVmap$ from $\mathcal{I}(Gr(K))$; that is, $I_{d+1, v}(K)$ is the intersection of $\mathcal{I}(Gr(K))$ with $\RR[\pp_J : J\in{\overline{\mathcal{F}}(K)}]$.
\label{DEF:Grasssection}
\end{definition}

\begin{remark} Since $\GRVmap$ implicitly depends on the choice of $\mathcal{B}_K$, so does $I_{d+1,v}(K)$. Hence we could have several different section ideals for the same cone. \end{remark}

The following lemma is an immediate consequence of  Definitions~\ref{DEF:GrassK} and ~\ref{DEF:Grasssection}.

\begin{lemma} \label{LEM:sectionvariety}
The closure of  the projection of the Grassmannian of $K$ onto its nonzero coordinates is the variety of the Grassmannian section ideal of $K$; that is,
\[
\overline{\pi_{\overline{\mathcal{F}}(K)}(Gr(K))} = \VV(I_{d+1,v}(K)).
\]
\end{lemma}

Lemma~\ref{LEM:sectionvariety} together with Theorem~\ref{THM:slackGrassequiv} gives us a geometric relationship between the slack variety and the Grassmannian section variety, namely that points in  $\VV(I_{d+1,v}(K))^*$ are in one-to-one correspondence with representatives of column scaling equivalence classes of~$\VV(I_P)^*$.
However, we also find the following algebraic relationship that allows us to consider the Grassmannian section ideal as a potentially simplifying algebraic relaxation  of the slack ideal.

\begin{proposition}
When appropriate Pl\"ucker variables are substituted for slack variables in the slack ideal $I_{K}$, the resulting ideal is contained in the Grassmannian section ideal $I_{d+1,v}(K)$.
\label{PROP:sectionideal}
\end{proposition}

\begin{proof}
Fix a choice of $\GRVmap$ for $K$. We have already seen that for each $\LL\in Gr(d+1,v)$, the matrix $\GRVmap(\LL) = S_K(\LL_{\BB})$ has rank~$d+1$. This means that the polynomials given by the $(d+2)$-minors of the symbolic matrix $\GRVmap(\pp)$ vanish on every point of $\LL\in Gr(d+1,v)$, and hence, these minors must be in the Pl\"ucker ideal $I_{d+1,v}$. Setting appropriate Pl\"ucker variables to zero in the minors and in $I_{d+1,v}$ preserves this containment, and gives the desired result after saturation.
\end{proof}

\begin{example} Recall that for a cone over a pentagon, $\mathcal{F}(K) = \emptyset$, and all the Pl\"ucker variables are used to fill a slack matrix
\[
\GRVmap(\pp) = \begin{bmatrix}
   0 &  p_{123} & p_{134} &  p_{145} &  0 \\
   0  & 0 &  p_{234} &  p_{245}  & p_{125} \\
    p_{123} & 0 &  0 & p_{345} &  p_{135} \\
   p_{124} &  p_{234}  & 0 & 0 &  p_{145} \\
  p_{125}  & p_{235} &  p_{345}  & 0  & 0
  \end{bmatrix}.
\]
So in this case, the Grassmannian section ideal is just the Pl\"ucker ideal.

Recall from Example~\ref{EX:pentagonslack} that the slack ideal of the pentagon had 25 generators of degree 4. In contrast, the Grassmannian section ideal, which was simply the Pl\"ucker ideal $I_{3,5}$, is much simpler having only the following $5$ trinomial generators of degree $2$:
\[
\begin{array}{c}
p_{235}p_{145}-p_{135}p_{245}+p_{125}p_{345} \\
p_{234}p_{145}-p_{134}p_{245}+p_{124}p_{345} \\
p_{234}p_{135}-p_{134}p_{235}+p_{123}p_{345} \\
p_{234}p_{125}-p_{124}p_{235}+p_{123}p_{245} \\
p_{134}p_{125}-p_{124}p_{135}+p_{123}p_{145}. \\
\end{array}
\]
In this case, $I_{3,5}$ coincides with the saturation of the ideal of $4$-minors of $\GRVmap(\pp)$.
\label{EX:sectionIdealPentagon}
\end{example}

\begin{remark} It might seem that the effect of substituting Pl\"ucker variables into the slack matrix is simply to force certain entries to be equal. While certain equalities are forced by the choice of $\GRVmap$, the insistence  that the entries come from a Pl\"ucker vector is in fact more restrictive than the slack matrix rank condition together with these equalities. That is, the containment of Proposition~\ref{PROP:sectionideal} is strict in general.
\end{remark}

\begin{example} Let $K$ be a $4$-dimensional cone over a triangular prism as in Example~\ref{EX:triprismslack}, with $\BB_K = \{123,456,124,136,236\}$. Then
\[
\GRVmap(\pp) = \begin{blockarray}{c*{5}{@{\;}c@{\;}}}
 & {123} & {456} & {124} & {136} & {236} \\
 \begin{block}{c@{\;\;}[*{5}{@{\;}c@{\;}}]}
1 &  0 &  p_{1456} & 0 &  0 &  p_{1236} \\
 2 & 0 &  p_{2456} & 0 &  p_{1236}  & 0\\
 3 & 0 &  p_{3456} & p_{1234} & 0 &  0 \\
 4 & p_{1234} &  0  & 0 & -p_{1346} &  p_{2346} \\
 5 & p_{1235}  & 0 &  -p_{1245}  & -p_{1356}  & p_{2356} \\
 6 & p_{1236}  & 0 &  -p_{1246}  & 0  & 0 \\
 \end{block}
 \end{blockarray}\;,
\]
and
\[
\begin{array}{rl}
 I_{4,6}(K)\! =\! & \!\!\!\!\!\langle\, p_{2346}p_{1456}+p_{1246}p_{3456}, p_{2346}p_{1356}+p_{1236}p_{3456}, p_{1246}p_{1356}-p_{1236}p_{1456},\\
& \!\!\! p_{1234}p_{1356}p_{2456}-p_{1235}p_{1246}p_{3456}, p_{1235}p_{1246}p_{2346}+p_{1234}p_{1236}p_{2456}\, \rangle.
\end{array}
\]
Simply setting the appropriate slack entries equal, we get
\[
S_K(\xx) = \begin{bmatrix}
0 &  x_1 & 0 &  0 &  x_2 \\
0 &  x_3 & 0 &  x_2  & 0\\
0 &  x_5 & x_6 & 0 &  0 \\
x_6 &  0  & 0 & 0 &  x_8 \\
x_9  & 0 &  0  & x_{10}  & 0 \\
x_2 & 0 &  x_{12}  & 0  & 0 \\
\end{bmatrix}
\]
and the resulting ideal is
\[
\left\langle \begin{array}{c} x_1x_8-x_5x_{12},\\ x_3x_6x_{10}-x_5x_9x_{12} \end{array}
\right\rangle \mapsto
\left\langle \begin{array}{c} p_{2346}p_{1456}+p_{1246}p_{3456},\\ p_{1234}p_{1356}p_{2456}-p_{1235}p_{1246}p_{3456} \end{array} \right\rangle \subsetneq I_{4,6}(K).
\]
\label{EX:triprismPlslack}
\end{example}

We have seen that the Grassmannian section ideal is in a sense a super ideal of the slack ideal (and a potentially simpler ideal) which cuts out a variety which corresponds to the ``section'' of the slack variety that contains an equivalence class representative whose entries come exactly from the Pl\"ucker coordinates of an element of $Gr(K)$. This is a direct consequence of Lemma~\ref{LEM:sectionvariety} and Theorem~\ref{THM:slackGrassequiv}.

\begin{corollary} The nonzero part of the Grassmannian section variety of $K$, $\VV(I_{d+1,v}(K))^*$, is in one-to-one correspondence with equivalence class representatives of $\VV(I_K)^*/\RR^f$ of the form $\GRVmap(\Lambda)$ for $\Lambda\in Gr(K)$. \label{COR:sectionvariety} \end{corollary}

\section{The Gale model} \label{sec:galemodel}

Here we introduce one final realization space, the space of {\em Gale transforms} of a polytope. Gale transforms (and their normalized counterparts, Gale diagrams) were developed by Perles in the 1960s (recorded by Gr\"unbaum \cite[Section 5.4]{Grunbaum}), and have long found use in the study of polytopes. 

Given a realization $Q = \conv\{\qq_1,\ldots, \qq_v\}$ of $d$-polytope $P$, let $\mat{B}$ be a matrix whose columns form a basis for the space of affine dependencies among the vertices of $Q$; that is,
\[
\begin{bmatrix} 1 & \cdots & 1 \\ \qq_1 & \cdots & \qq_v \end{bmatrix} \cdot \mat{B} = 0, \quad \mat{B}\in\RR^{v\times (v-d-1)},\; \rk(\mat{B})=v-d-1.
\]
Let $\bb_i^\top$ be the $i$th row of $\mat{B}$. The {\em Gale transform} of $Q$  is the vector configuration $\GG = \{\bb_1,\ldots, \bb_v\} \subset \RR^{v-d-1}$. Recall by definition of a slack matrix $S_Q$, this implies
\[
\begin{bmatrix}\bb_1 & \cdots & \bb_v \end{bmatrix}\cdot S_Q = 0.
\]
Notice that for any $\mat{A}\in GL_{v-d-1}(\RR)$, the configuration $\{\mat{A}\bb_1,\ldots, \mat{A}\bb_v\}$ is also a Gale transform of $Q$, since $\mat{B}$ and $\mat{BA}^\top$ have the same column space.

\begin{definition}
By slight abuse of terminology, for each $\ss\in\VV(I_P)$ call a vector configuration $\GG_{\ss} = \{\bb_1,\ldots, \bb_v\}$ a {\em Gale transform} of $S_P(\ss)$ if the matrix $\mat{B}_{\ss}:=\begin{bmatrix}\bb_1 & \cdots & \bb_v \end{bmatrix}^\top$ is full rank and $\begin{bmatrix}\bb_1 & \cdots & \bb_v \end{bmatrix} S_P(\ss) = 0$.
\label{DEF:gale}\end{definition}

Notice that after normalizing the $\bb_i$'s to have unit length, two distinct Gale transforms may correspond to the same Gale diagram. This normalization corresponds to scaling the rows of slack matrix $S_Q$ which may change its column space, which is why we will work with Gale transforms and not Gale diagrams.

\begin{example} Let $P$ be the (abstract) triangular prism with facets $123$, $456$, $1245$, $1346$, and $2356$. Consider the element of $\VV(I_P)^*$ given by
\[
S_P(\ss) = \begin{bmatrix} 0 & 1 &  0 &  0 &  1 \\ 0 &  1 &  0 &  1 &  0 \\ 0 &  1 &  1 &  0 &  0 \\ 1 &  0 &  0 &  0 &  1 \\ 2 &  0 &  0 &  2 &  0 \\ 1 &  0 &  1 &  0 &  0 \end{bmatrix}.
\]
It is not hard to check that $\mathbbm{1}\notin\rho(S_P(\ss))$, so that by \cite[Theorem 2.2]{GMTWfirstpaper}, $S_P(\ss)$ is not a slack matrix of a realization of $P$, but it has a Gale transform given by the rows of $\mat{B}_{\ss}$ and pictured below.\\
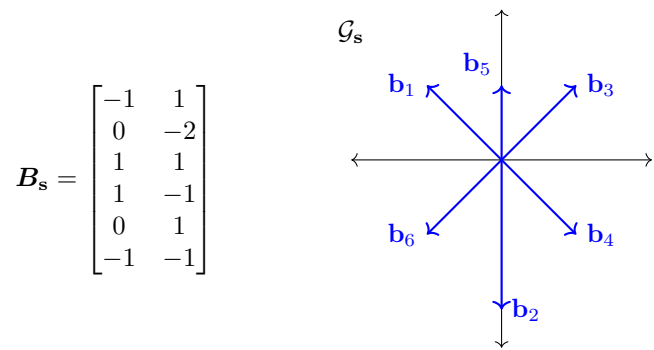
\begin{center}
\begin{minipage}{0.45\textwidth}
\[
\mat{B}_{\ss} = \begin{bmatrix} -1 & 1 \\ 0 & -2 \\ 1 & 1 \\ 1 & -1 \\  0 & 1 \\ -1 & -1 \end{bmatrix}
\]
\end{minipage}\begin{minipage}{0.45\textwidth}
\begin{tikzpicture}
\draw [<->] (-2,0) -- (2,0);
\draw [<->] (0,-2.5) -- (0,2);
\draw (-1,1) node[inner sep = 0pt, minimum size = 0pt, label=left:$\textcolor{blue}{\bb_1}$] (b1) {};
\draw (0,-2) node[inner sep = 0pt, minimum size = 0pt, label=right:$\textcolor{blue}{\bb_2}$] (b2) {};
\draw (1,1) node[inner sep = 0pt, minimum size = 0pt, label=right:$\textcolor{blue}{\bb_3}$] (b3) {};
\draw (1,-1) node[inner sep = 0pt, minimum size = 0pt, label=right:$\textcolor{blue}{\bb_4}$] (b4) {};
\draw (0,1) node[inner sep = 0pt, minimum size = 0pt, label=160:$\textcolor{blue}{\bb_5}$] (b5) {};
\draw (-1,-1) node[inner sep = 0pt, minimum size = 0pt, label=left:$\textcolor{blue}{\bb_6}$] (b6) {};
\draw [blue,thick, ->] (0,0) -- (b1);
\draw [blue, thick, ->] (0,0) -- (b2);
\draw [blue,thick, ->] (0,0) -- (b3);
\draw [blue,thick, ->] (0,0) -- (b4);
\draw [blue,thick, ->] (0,0) -- (b5);
\draw [blue, thick,->] (0,0) -- (b6);
\draw (-2,1.7) node[] {$\GG_{\ss}$};
\end{tikzpicture}
\end{minipage}\end{center}
\label{EX:prismGaleext}
\end{example}

\begin{remark} One can easily check from the Gale transform $\GG_{\ss} = \{\bb_1,\ldots, \bb_v\}$ whether it comes from an actual realization of $P$, since $\mathbbm{1}\in\rho(S_P(\ss))$ if and only if $\bb_1+\cdots+\bb_v = 0$, by definition of $\GG_{\ss}$.
\end{remark}

Denote by $\GG(P)$ the set of all possible Gale transforms of a given abstract polytope $P$; that is,
\[
\GG(P) := \{\GG_{\ss} : \ss\in\VV(I_P)^*\}.
\]
Of course, for $K=P^h$ we already have $\VV(I_K)^*=\VV(I_P)^*$, so that
\begin{equation} \GG(K) := \{\GG_{\ss} : \ss\in\VV(I_K)^*\} = \GG(P). \label{EQ:genGale} \end{equation}

From Proposition~\ref{PROP:colscale} and the definition of Gale transform, it is clear that there is a one-to-one correspondence between $GL_{v-d-1}(\RR)$ equivalence classes of $\GG(K)$ and elements of~$\VV(I_K)^*$ up to column scaling.
For this reason, and since for fixed~$\ss$, each~$\GG_{\ss}$ just comes from a different choice of basis for the column space~$\rho(\mat{B}_{\ss})$,
it makes sense to consider $\GG(K)$ in the Grassmannian $Gr(v-d-1,v)$. Once again we consider how the combinatorics of $K$ is encoded in the Pl\"ucker coordinates of each $\GG_{\ss}\in\GG(K)$.

\begin{lemma} Let $\ss\in\VV(I_K)$. A set of rows $J\subsetneq [v]$ of $S_K(\ss)$ is dependent if and only if $[v]\backslash J$ indexes a set of rows of its Gale transform $\mat{B}_{\ss}$ which are dependent.
\label{LEM:Galeduality}
\end{lemma}

\begin{remark} This lemma is weaker than the usual characterization of the combinatorics of a polytope from its Gale transform, namely that $J$ indexes a face of $P$ if and only if $J=[v]$ or if~$\mathbf{0}$ is in the relative interior of Gale vectors $\{\bb_j:j\in[v]\backslash J\}$.
This lemma is also a consequence of the fact that matrices whose columns form bases for orthogonal vector spaces define dual matroids (see \cite[Lemma 1.4.5]{W19} for a proof of Lemma~\ref{LEM:Galeduality}).
\end{remark}

\begin{example} Continuing Example~\ref{EX:prismGaleext}, the only sets of non-trivially dependent rows of $S_P$ are $\{1,2,4,5\}$, $\{1,3,4,6\}$, and $\{2,3,5,6\}$, which correspond to facets of~$P$. Looking at the drawing of the Gale transform in Example \ref{EX:prismGaleext}, we can see that the only dependent pairs of vectors are $\{3,6\}$, $\{2,5\}$ and $\{1,4\}$, as expected.

Notice that for this example, we still have $\mathbf{0}$ in the relative interior of $\{\bb_j:j\in[6]\backslash J\}$ for faces $J$ of $P$. However if we consider the slack matrix corresponding to $\ss'\in\VV(I_P)^*$, as given below, then its Gale transform does not have $\mathbf{0}\in\textup{rel int}\{\bb_2,\bb_5\}$ even though $\{1,3,4,6\}$ represents a facet.

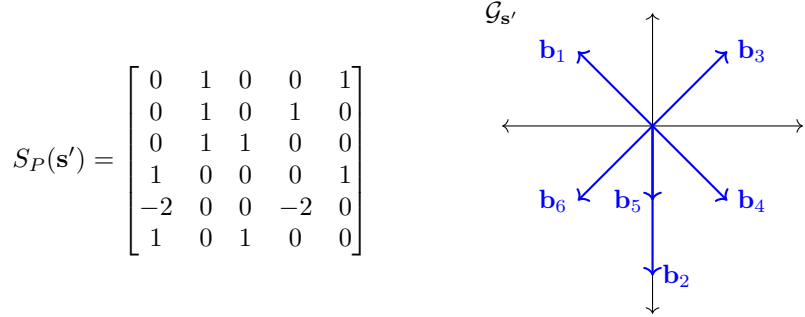
\begin{center}
\begin{minipage}{0.4\textwidth} \[
S_P(\ss') = \begin{bmatrix} 0 & 1 &  0 &  0 &  1 \\ 0 &  1 &  0 &  1 &  0 \\ 0 &  1 &  1 &  0 &  0 \\ 1 &  0 &  0 &  0 &  1 \\ -2 &  0 &  0 &  -2 &  0 \\ 1 &  0 &  1 &  0 &  0 \end{bmatrix}
\]
\end{minipage}\hspace{0.1\textwidth}\begin{minipage}{0.4\textwidth}
\begin{tikzpicture}
\draw [<->] (-2,0) -- (2,0);
\draw [<->] (0,-2.5) -- (0,1.5);
\draw (-1,1) node[inner sep = 0pt, minimum size = 0pt, label=left:$\textcolor{blue}{\bb_1}$] (b1) {};
\draw (0,-2) node[inner sep = 0pt, minimum size = 0pt, label=right:$\textcolor{blue}{\bb_2}$] (b2) {};
\draw (1,1) node[inner sep = 0pt, minimum size = 0pt, label=right:$\textcolor{blue}{\bb_3}$] (b3) {};
\draw (1,-1) node[inner sep = 0pt, minimum size = 0pt, label=right:$\textcolor{blue}{\bb_4}$] (b4) {};
\draw (0,-1) node[inner sep = 0pt, minimum size = 0pt, label=left:$\textcolor{blue}{\bb_5}$] (b5) {};
\draw (-1,-1) node[inner sep = 0pt, minimum size = 0pt, label=left:$\textcolor{blue}{\bb_6}$] (b6) {};
\draw [blue,thick, ->] (0,0) -- (b1);
\draw [blue, thick, ->] (0,0) -- (b2);
\draw [blue,thick, ->] (0,0) -- (b3);
\draw [blue,thick, ->] (0,0) -- (b4);
\draw [blue,thick, ->] (0,0) -- (b5);
\draw [blue, thick,->] (0,0) -- (b6);
\draw (-2,1.5) node[] {$\GG_{\ss'}$};
\end{tikzpicture}
\end{minipage}\end{center}
\end{example}

Notice that this example shows that even after normalizing the $\bb_i$ we may not get the Gale diagram of a polytope. This is because $\GG(P)$ contains Gale transforms in which we allow negative scalings of the vectors, not just positive ones.

Using Lemma \ref{LEM:Galeduality}, we characterize the Gale transforms of combinatorial type $K$ in the Grassmannian as follows.

\begin{definition}
Define the  {\em dual Grassmannian of cone $K$} to be
\begin{align*}
Gr^*(K) := \left\{ \begin{array}{cc}
\parbox[c][][c]{3.4cm}{$\LL\in Gr(v-d-1,v)$\ :} &  \parbox[c][][c]{4.2cm}{$pl(\LL)_{[v]\backslash J} = 0 \ \forall J \in \mathcal{F}(K)$,\\ $pl(\LL)_{[v]\backslash J} \neq 0 \ \forall J \in \overline{\mathcal{F}}(K)$} \end{array}
\right\}
\end{align*}
\end{definition}

Notice that $Gr^*(K)$ and $Gr(K)$ are isomorphic under the standard isomorphism of Grassmannians which sends a subspace $\LL$ to its orthogonal complement $\LL^\perp$.
\begin{equation}
\begin{array}{rcl}
Gr(d+1,v) & \cong & Gr(v-d-1,v) \\
\LL & \leftrightarrow & \LL^\perp.
\end{array} \label{EQ:grassIso}\end{equation}
We will use this isomorphism in the proof of the following proposition which shows that $Gr^*(K)$ actually captures exactly the desired Gale transforms. 

\begin{proposition} There is a one-to-one correspondence between elements of $Gr^*(K)$ and elements of $\GG(K)$ modulo the action of $GL_{v-d-1}(\RR)$.
\label{PROP:galeGrassequiv}\end{proposition}

\begin{proof} Given an element $\GG_{\ss}\in\GG(K)$, we map it to the column space of $\mat{B}_{\ss}$. This space is an element of $Gr(v-d-1,v)$ since $\mat{B}_{\ss}$ is full rank and satisfies the conditions of $Gr^*(K)$ by Lemma~\ref{LEM:Galeduality}.

To see that each element of $Gr^*(K)$ represents the Gale transform of something in $\VV(I_K)^*$, we use \eqref{EQ:grassIso}.
As a map on Pl\"ucker coordinates this translates to
\begin{eqnarray}
\big(pl(\LL)_J\big)_{J\in\binom{[v]}{d+1}} & \leftrightarrow & \big({\rm sgn}(J)\cdot pl(\LL)_{[v]\backslash J}\big)_{J\in\binom{[v]}{d+1}}, \label{EQ:dualPl}
\end{eqnarray}
where ${\rm sgn}(J)$ is the sign of the permutation $(J, [v]\setminus J)$.
The result now follows from the fact that elements of $Gr(K)$ are in one-to-one correspondence with elements of $\VV(I_K)^*$, up to column scaling by Theorem~\ref{THM:slackGrassequiv}.
\end{proof}

\begin{example} Recall that for a cone $K$ over a pentagon, all sets $J\in \binom{[v]}{d+1}$ are facet extensions and hence
\[
Gr(K) = \{\LL\in Gr(3,5) : pl(\LL)_J \neq 0\quad \forall J\in\binom{[5]}{3}\}.
\]
From this we also get
\[
Gr^*(K) = \{\LL\in Gr(2,5) : pl(\LL)_J \neq 0\quad \forall J\in\binom{[5]}{2}\}.
\]
So for example, the matrix $\mat{B}$ below represents an element of $Gr^*(K)$ and hence a Gale transform of some $\ss\in\VV(I_K)^*$.
\begin{center}
\begin{minipage}{0.4\textwidth}
\[
\mat{B} = \begin{bmatrix} 1 & 1 \\ 1 & -1 \\ 0 & 1 \\ 2 & 0 \\ -1 & 2 \end{bmatrix}
\]
\end{minipage}\hspace{0.1\textwidth}\begin{minipage}{0.4\textwidth}
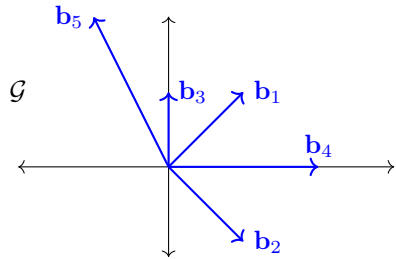
\begin{tikzpicture}
\draw [<->] (-2,0) -- (3,0);
\draw [<->] (0,-1.2) -- (0,2);
\draw (1,1) node[inner sep = 0pt, minimum size = 0pt, label=right:$\textcolor{blue}{\bb_1}$] (b1) {};
\draw (1,-1) node[inner sep = 0pt, minimum size = 0pt, label=right:$\textcolor{blue}{\bb_2}$] (b2) {};
\draw (0,1) node[inner sep = 0pt, minimum size = 0pt, label=right:$\textcolor{blue}{\bb_3}$] (b3) {};
\draw (2,0) node[inner sep = 0pt, minimum size = 0pt, label=above:$\textcolor{blue}{\bb_4}$] (b4) {};
\draw (-1,2) node[inner sep = 0pt, minimum size = 0pt, label=left:$\textcolor{blue}{\bb_5}$] (b5) {};
\draw [blue,thick, ->] (0,0) -- (b1);
\draw [blue, thick, ->] (0,0) -- (b2);
\draw [blue,thick, ->] (0,0) -- (b3);
\draw [blue,thick, ->] (0,0) -- (b4);
\draw [blue,thick, ->] (0,0) -- (b5);
\draw (-2,1) node[] {$\GG$};
\end{tikzpicture}
\end{minipage}\end{center}
To see to which element of the slack variety this Gale transform corresponds we recall that given $\LL\in Gr(K)$, we get a slack matrix via the map $\GRVmap$, which gives us the matrix in Example \ref{EX:sectionIdealPentagon}.

Using \eqref{EQ:dualPl}, we easily obtain a slack matrix without constructing $\LL$:
\[
S_K = \begin{bmatrix}
 0 &  p_{45} & p_{25} &  p_{23} &  0 \\
 0  & 0 &  -p_{15} &  -p_{13}  & p_{34} \\
 p_{45} & 0 &  0 & p_{12} &  -p_{24} \\
 -p_{35} &  -p_{15}  & 0 & 0 &  p_{23} \\
 p_{34}  & p_{14} &  p_{12}  & 0  & 0
\end{bmatrix} =
\begin{bmatrix}
0 & 4 & 1 & 1 & 0\\ 0 & 0 & -3 & -1 & -2\\ 4 & 0 & 0 & -2 & -2\\ -1 & -3 & 0 & 0 & 1\\ -2 & -2 & -2 & 0 & 0
\end{bmatrix}.
\]
\label{EX:pentagondualGrass}
\end{example}

As a simple corollary, we get a bijection from the Gale (dual Grassmannian) space to the slack variety (see Figure~\ref{FIG:relatemodelsmaster}).

\begin{corollary}
There is a birational equivalence between $Gr^*(K)$ and $\VV(I_K)^*$ up to column scaling.
\label{COR:slackDgrassequiv}
\label{COR:dualGrassSlackEquiv}
\end{corollary}

\begin{remark} For polytopes, it is well-known that two realizations $Q,Q'$ are projectively equivalent if and only if their Gale transforms $\GG, \GG'$ are related by a linear transformation  and positive scaling of the vectors; that is, $\bb_i' = \lambda_i\mat{A}\bb_i$ for $\mat{A}\in GL_{v-d-1}(\RR)$, $\lambda_i\in\RR_{>0}$. An analogous result holds for our generalized Gale transforms. Two elements of $\VV(I_K)^*$  differ by row and column scaling if and only if their Gale transforms are related by a linear transformation and nonzero scaling of the vectors.
\end{remark}


%
%


\section{The reduced slack model} \label{sec:applications}

In this section, we take inspiration from the Grassmannian model in order to streamline the computations associated to the slack model in some cases. In the Grassmannian model we take any set of columns that generate the column space of a slack matrix and we saw that it preserves essentially all the information from the polytope. This suggests that we may not need to work with the full slack matrix, as we could recover the rest from any sufficiently large submatrix.

The tricky part is to guarantee that the same holds symbolically, i.e., if we simply have a submatrix of the symbolic slack matrix and impose the rank constraints, do we still expect to be able to recover the full symbolic slack matrix with the right support? We will explore this question, give some sufficient conditions for this idea to work, and present some examples of its application in order to illustrate how powerful the insights carried from one model to the others can be.

Recall that we denote the nonzero part of the slack variety by $\VV(I_P)^* := \VV(I_P)\cap (\RR^*)^t$. Notice that, by \cite[Corollary 1 p. 84]{M88} the slack variety coincides with the complex closure of $\VV(J)-\VV(\langle x_1,\ldots, x_t\rangle)$, the minor variety without coordinate hyperplanes, and hence $\overline{\VV(I_P)^*}=\VV(I_P)$.

\begin{lemma}
Let $F$ be a set of facets of a $d$-dimensional polytope $P$ containing a flag (so that, as in the proof of Proposition \ref{PROP:colscale}, $S_P$ contains a ($d+1$)-triangular  submatrix in columns $F$). Then there exists a birational map from the nonzero part of the slack variety $\VV(I_P)^*$ up to column scaling to the projection $\pi_F(\VV(I_P)^*)$ of $\VV(I_P)^*$ onto the columns of~$F$ up to column scaling.
\label{LEM:slackProjequivalence}
\end{lemma}

\begin{proof}
Let us consider the following diagram:
\begin{center}
\begin{tikzpicture}
\node (a) at (0,0) {$\VV(I_P)^*/\RR^f$};
\node (b) at (3.5,0) {$\pi_F(\VV(I_P)^*)/\RR^{|F|}$};
\node (c) at (0,-2) {$Gr(K)$};
\path[->,font=\scriptsize,>=angle 90]([xshift= 2pt]c.north) edge node[right] {$\GRVmap$} ([xshift= 2pt]a.south)
([xshift= -2pt]a.south) edge node[left] {$\VGRmap$} ([xshift= -2pt]c.north);
\path[->](a) edge node[above]{$\pi_F$} (b)
(b) edge node[below]{$\rho$} (c);
\end{tikzpicture}
\end{center}
where $\rho$ is the map that associates to a matrix its column space.

Notice that, since $F$ contains a flag we have $\rho \circ \pi_F = \VGRmap$, so that by Lemma~\ref{LEM:inverses} we see $\GRVmap \circ \rho \circ \pi_F = id_{\VV(I_P)^*/\RR^f}$, and then $\pi_F \circ \GRVmap \circ \rho \circ \pi_F = \pi_F$. Hence, $\pi_F \circ \GRVmap \circ \rho = id_{\pi_F(\VV(I_P)^*)/\RR^{|F|}}$ since $\pi_F$ is surjective.
\end{proof}

\begin{definition}
Let $F$ be a set of facets of a $d$-dimensional abstract polytope $P$ such that:
\begin{enumerate}
\item $F$ contains a flag,
\item all facets of $P$ not in $F$ are simplicial.
\end{enumerate}
Let $S_F$ be the submatrix of $S_P$ consisting of only the columns indexed by $F$. We call $S_F$ a \textit{reduced slack matrix} for $P$.

Let $I_F$ be the slack ideal of the reduced symbolic slack matrix $S_F(\xx)$.
Denote the nonzero part of its slack variety $\VV(I_F)$ 
by $\VV_F$; that is,
\[
\VV_F := \left\{S\in\RR^{v\times |F|} : \begin{array}{l} S \text{ has zero pattern of $S_P$ restricted to $F$} \\ \text{and }\rk(S) = d+1 \end{array}\right\}.
\]
\label{def:reducedSlackMatrix}
\end{definition}

\begin{remark} Contrary to what one might initially assume, $\VV_F$ is not the same as the projection, $\pi_F(\VV(I_P)^*)$, of the nonzero part of the slack variety onto $F$. The projection $\pi_F(\VV(I_P)^*)$ is also subject to the additional condition that each of the determinants of the projected matrix that is used to fill the remaining columns of the matrix (via the map $\GRVmap$) must also be nonzero.
In the following example we show that $\pi_F(\VV(I_P)^*)\subsetneq\VV_F$.
\label{REM:projInVF}
\end{remark}

\begin{example}
Let us consider the first nonpolytopal $3$-sphere in \cite[Table 2]{AS85} whose facets have the following vertex sets:
\[
12345, 12346, 12578, 12678, 14568, 34578, 2357, 2367, 3467, 4678.
\]

If there was a polytope $P$ with these facets, its symbolic slack matrix would be
\[
S_P(\xx)=\begin{bmatrix}
    0  &        0 &       0 &       0 &       0 &   x_1   &    x_2  &    x_3   & x_4    & x_5      \\
    0  &        0 &       0 &       0 &     x_6 &   x_7   &      0  &      0   & x_8    & x_9      \\
    0  &        0 &  x_{10} &  x_{11} &  x_{12} & 0       &      0  &      0   &   0    & x_{13}   \\
    0  &        0 &  x_{14} &  x_{15} &  0      & 0       & x_{16}  & x_{17}   &   0    & 0        \\
    0  &   x_{18} &       0 &  x_{19} &  0      & 0       &      0  & x_{20}   & x_{21} & x_{22}   \\
x_{23} &        0 &  x_{24} &       0 &  0      & x_{25}  & x_{26}  &      0   &   0    & 0        \\
x_{27} &   x_{28} &       0 &       0 &  x_{29} & 0       &      0  &      0   &   0    & 0        \\
x_{30} &   x_{31} &       0 &       0 &  0      & 0       & x_{32}  & x_{33}   & x_{34} & 0        \\
\end{bmatrix}.
\]
Let $F$ be the set consisting of the first $6$ facets of $P$ and let $S_F(\xx)$ be the corresponding reduced (symbolic) slack matrix. Then its slack ideal is
\[
I_F = \left\langle \begin{array}{c} x_{28}x_{30} - x_{27} x_{31}, x_{15} x_{18} x_{24} x_{30} - x_{14} x_{19} x_{23} x_{31}, \\
x_{11} x_{14} - x_{10} x_{15}, x_{15} x_{18} x_{24} x_{27} - x_{14} x_{19} x_{23} x_{28},\\
x_{11} x_{18} x_{24} x_{30} - x_{10} x_{19} x_{23} x_{31}, x_{11} x_{18} x_{24} x_{27} - x_{10} x_{19} x_{23} x_{28}
\end{array}
\right\rangle.
\]
Then we extend the slack matrix $S_F(\xx)$ to the full matrix $S_P(\xx)$ using the map $\GRVmap$, filling in the remaining entries with the corresponding Pl\"ucker coordinates. More precisely, we write the following polynomials in place of the variables $x_i$ with $i=2,3,4,5,8,9,13,16,17,20,21,22,26,32,33,34$:
{\allowdisplaybreaks \begin{gather*}
x_1x_6x_{11}x_{18}x_{27},\ \ -x_1x_6x_{11}x_{23}x_{28},\ \ -x_1x_{12}x_{15}x_{23}x_{28},\ \ -x_1x_{15}x_{23}x_{29}x_{31}, \\
-x_7x_{12}x_{15}x_{23}x_{28},\ \ x_6x_{15}x_{25}x_{28}x_{30}-x_6x_{15}x_{25}x_{27}x_{31}-x_7x_{15}x_{23}x_{29}x_{31}, \\
x_{12}x_{15}x_{25}x_{28}x_{30}-x_{12}x_{15}x_{25}x_{27}x_{31},\ \ x_7x_{12}x_{15}x_{18}x_{27},\ \ -x_7x_{12}x_{15}x_{23}x_{28}, \\
-x_6x_{11}x_{18}x_{25}x_{27}-x_7x_{12}x_{19}x_{23}x_{28}-x_7x_{11}x_{18}x_{23}x_{29},\ \ -x_{12}x_{15}x_{18}x_{25}x_{27},\\
-x_{15}x_{18}x_{25}x_{29}x_{30},\ \ x_6x_{11}x_{18}x_{25}x_{27}+x_7x_{12}x_{19}x_{23}x_{28}+x_7x_{11}x_{18}x_{23}x_{29}, \\
x_7x_{12}x_{19}x_{28}x_{30}+x_7x_{11}x_{18}x_{29}x_{30}-x_7x_{12}x_{19}x_{27}x_{31},\\
x_6x_{11}x_{25}x_{28}x_{30}-x_6x_{11}x_{25}x_{27}x_{31}-x_7x_{11}x_{23}x_{29}x_{31},\\
x_{12}x_{15}x_{25}x_{28}x_{30}-x_{12}x_{15}x_{25}x_{27}x_{31}.
\end{gather*}}\\[-10pt]
Call $f$ the last of these entries. Computing the saturation of $I_F$ with respect to~$f$ results in the trivial ideal $\langle 1 \rangle$, which tells us that the $3$-sphere is not realizable as a polytope. In other words, on the variety of $I_F$, the map $\GRVmap$ will always set entry~$x_{34}$ to zero.

Moreover, in this case $\overline{\pi_F(\VV(I_P))} \subsetneq \overline{\VV_F}$. One can check that $I_P = \langle 1\rangle$, so that ${\pi_F(\VV(I_P))} = \emptyset$, but the reduced slack matrix with all variables set to $1$, for example, lies in $\VV_F$. As we show in Lemma \ref{LEM:closureProjection}, the difference between these two varieties is the variety of an ideal generated by the product of some $(d+1)$-minors of $S_F(\xx)$. Denote this ideal by $K_F$, then more precisely, in the above example the ideal $K_F$ is generated by the product of all $5$-minors whose column indices are $\{1,2,4,5,6\}$ (corresponding to a flag) and whose row indices are of the form $\{i\} \cup J$, where $i=1,\dots,8$ and $J \in \{\{2,3,5,7\}, \{2,3,6,7\}, \{3,4,6,7\}, \{4,6,7,8\}\}$ come from the vertex sets of the simplicial facets not in $F$.

Notice that, $\dim I_F = 31-16=15$ in the polynomial ring in the variables of~$I_F$, since there are $16$ variables in the columns outside $F$. However, since $S_F(\xx)$ is a $8 \times 6$ matrix, we can scale $8+6-1=13$ variables (by \cite[Lemma 5.2]{GMTWsecondpaper}), that leave us $15-13=2$ degrees of freedom.
\label{EX:nonrealizableSphere}
\end{example}

\begin{remark}
Given an abstract polytopal sphere $P$, we can ask two different questions: checking whether $\VV(I_P) \neq \emptyset$ is an algebraic question that concerns realizability of $P$ as a matroid, as in Example \ref{EX:nonrealizableSphere}; proving that $\VV_+(I_P) \neq \emptyset$ is a semialgebraic question that asks about the realization of $P$ as a polytope. Usually checking non-emptyness of semialgebraic sets is harder. Realizability of spheres is a classical problem: polytopal and non-polytopal $3$-spheres with eight vertices were enumerated by Altshuler and Steinberg \cite{AS84,AS85}, those with nine vertices were enumerated more recently by Firsching \cite{F20}. On the other hand, Bokowski and Sturmfels \cite{BS87} described an algorithm to check polytopality of certain spheres. However, there are still many spheres for which it is not known whether they are polytopal or not, see e.g., \cite[Section 5]{CS19} and \cite[Remark 5.2]{Z20}. For the matroid case similar realizability issues arise, as one can see in \cite{RS91,FMM13}.
\end{remark}

From now on, instead of studying $\pi_F(\VV(I_P)^*)$, we can focus on $\pi_F(\VV(I_P))$ since they have the same Zariski closure, i.e.
\begin{equation}
\overline{\pi_F(\VV(I_P))} = \overline{\pi_F(\VV(I_P)^*)}.
\label{EQ:closureVPvsVIP}
\end{equation}
This is a consequence of the fact $\VV(I_P) = \overline{\VV(I_P)^*}$ and of basic properties of the Zariski topology.

In general, let $K_F$ be the ideal defined as in Example~\ref{EX:nonrealizableSphere}. That is $K_F$ is the ideal generated by the product of the $(d+1)$-minors of $S_F(\xx)$ which are used to fill the additional columns of a slack matrix via the map $\GRVmap$.

\begin{lemma} 
Let $F$ satisfy the conditions of Definition~\ref{def:reducedSlackMatrix}.  
Then
\[
\overline{\pi_F(\VV(I_P))} = \overline{\VV_F\setminus \VV(K_F)}.
\]
\label{LEM:closureProjection}
\end{lemma}

\begin{proof}
By Remark \ref{REM:projInVF}, $\overline{\VV(I_P)^*} = \overline{\VV_F \setminus \VV(K_F)}$, and hence the claim follows by \eqref{EQ:closureVPvsVIP}.
\end{proof}

\begin{corollary}
The variety $\overline{\pi_F(\VV(I_P))}$ is the union of some of the irreducible components of $\overline{\VV_F}$. In particular, if $\overline{\VV_F}$ is irreducible, then either $\overline{\pi_F(\VV(I_P))}$ is empty or
\[
\overline{\pi_F(\VV(I_P))} = \overline{\VV_F}.
\]
\label{COR:closureProjection}
\end{corollary}

\begin{proof} First assume $\overline{\VV_F}$ is irreducible. Then writing
\[
\overline{\VV_F} = \left(\VV(K_F)\cap \overline{\VV_F}\right) \cup \overline{\overline{\VV_F} \setminus \VV(K_F)}
\]
tells us that either $\VV_F \subset \VV(K_F)$ or $\overline{\VV_F} = \overline{\overline{\VV_F} \setminus \VV(K_F) }$. By Lemma~\ref{LEM:closureProjection}, in the first case, we see that $\pi_F(\VV(I_P)) = \emptyset$, and in the second case $\overline{\VV_F} = \overline{\pi_F(\VV(I_P))}$.

A similar argument shows that if $\overline{\VV_F} = \cup_{i=1}^m W_i$ for some irreducible components $W_i$, then $\overline{\overline{\VV_F} \setminus \VV(K_F) } = \cup_{j=1}^s W_{i_j}$ for some $i_1,\ldots, i_s \in \{1,\ldots,  m\}$.
\end{proof}

Notice that, in Example \ref{EX:nonrealizableSphere}, $\overline{\VV_F}$ is irreducible over the rationals, and since the decomposition of the proof of Corollary~\ref{COR:closureProjection} is also over the rationals, we get that $\overline{\pi_F(\VV(I_P))} \subsetneq \overline{\VV_F}$ implies that $\pi_F(\VV(I_P)) = \emptyset$.

\begin{example} Consider the $3$-polytope $P$ which has $8$ vertices and facets $\{1357$,$1458$, $2367$, $2468$,$5678$,$134$,$234\}$, pictured below. Then its symbolic slack matrix is

\begin{minipage}{0.35\textwidth}
\begin{tikzpicture}
\draw (0,0) node[inner sep = 0pt, minimum size = 1pt, circle, draw, fill, label=below:$5$] (4) {};
\draw (0,2) node[inner sep = 0pt, minimum size = 1pt, circle, draw, fill, label=above:$1$] (0) {};
\draw (-1.6,1.4) node[inner sep = 0pt, minimum size = 1pt, circle, draw, fill, label=left:$3$] (2) {};
\draw (1.6,1.4) node[inner sep = 0pt, minimum size = 1pt, circle, draw, fill, label=right:$4$] (3) {};
\draw (0.4,0.1) node[inner sep = 0pt, minimum size = 1pt, circle, draw, fill, label=above:$2$] (1) {};
\draw (0.4,-1.8) node[inner sep = 0pt, minimum size = 1pt, circle, draw, fill, label=below:$6$] (5) {};
\draw (-1.6,-1) node[inner sep = 0pt, minimum size = 1pt, circle, draw, fill, label=left:$7$] (6) {};
\draw (1.6,-0.6) node[inner sep = 0pt, minimum size = 1pt, circle, draw, fill, label=right:$8$] (7) {};

\draw (0) -- (2) -- (3) -- (7) -- (5) -- (1) -- (2) -- (6) -- (5);
\draw (0) -- (3) -- (1);
\draw[dashed] (0) -- (4) -- (7);
\draw[dashed] (4) -- (6);
\end{tikzpicture}
\end{minipage}
\begin{minipage}{0.53\textwidth}
$$ S_P(\xx) \!=\! \left[\begin{array}{ccccccc}
0 & 0 & x_{1} & x_{2} & x_{3} & 0 & x_{4} \\
x_{5} & x_{6} & 0 & 0 & x_{7} & x_{8} & 0 \\
0 & x_{9} & 0 & x_{10} & x_{11} & 0 & 0 \\
x_{12} & 0 & x_{13} & 0 & x_{14} & 0 & 0 \\
0 & 0 & x_{15} & x_{16} & 0 & x_{17} & x_{18} \\
x_{19} & x_{20} & 0 & 0 & 0 & x_{21} & x_{22} \\
0 & x_{23} & 0 & x_{24} & 0 & x_{25} & x_{26} \\
x_{27} & 0 & x_{28} & 0 & 0 & x_{29} & x_{30}
\end{array}\right].
$$
\end{minipage}

Notice that the set of $5$ nonsimplicial facets $F$ contains a flag, hence a reduced slack matrix is
$$ S_F(\xx) = \left[\begin{array}{ccccccc}
0 & 0 & x_{1} & x_{2} & x_{3} \\
x_{5} & x_{6} & 0 & 0 & x_{7}  \\
0 & x_{9} & 0 & x_{10} & x_{11}  \\
x_{12} & 0 & x_{13} & 0 & x_{14}  \\
0 & 0 & x_{15} & x_{16} & 0  \\
x_{19} & x_{20} & 0 & 0 & 0  \\
0 & x_{23} & 0 & x_{24} & 0  \\
x_{27} & 0 & x_{28} & 0 & 0
\end{array}\right],
$$
which gives a reduced slack ideal $I_F$ in the polynomial ring\break $\RR[x_1, x_{2}, x_{3},x_{5},x_{6},x_{7},x_{9}, x_{10}, x_{11},x_{12},x_{13},x_{14},x_{15},x_{16},x_{19}, x_{20},x_{23},x_{24},x_{27},x_{28}]$ of dimension 16 with 41 generators. Notice that after scaling 12 rows and columns as in \cite[Lemma 5.2]{GMTWsecondpaper}, this results in a slack realization space of dimension 4. Now taking as our flag all nonsimplicial facets except $\{1,3,5,7\}$ (and, for clarity, setting some entries to one as remarked above) we see $\GRVmap(\rho(S_F(\xx)))$ equals
\[
\begingroup 
\setlength\arraycolsep{2pt}
\begin{bmatrix}
0 & 0 & x_1 & 1 & 1 & 0 & -x_1x_{10}+x_9+x_{10}-1 \\
x_5 & 1 & 0 & 0 & 1 & x_1x_{10}-x_9-x_{10}+1 & 0 \\
0 & x_9 & 0 & x_{10} & 1 & 0 & 0 \\
1 & 0 & 1 & 0 & 1 & 0 & 0 \\
0 & 0 & 1 & x_{16} & 0 & -x_1x_9x_{16}+x_9x_{16}+x_9 & x_9x_{16}-x_{10}-x_{16} \\
1 & x_{20} & 0 & 0 & 0 & x_1x_{10}x_{20}-x_{10}x_{20}+x_{20} & -x_{10}x_{20} \\
0 & x_{23} & 0 & 1 & 0 & x_1x_{10}x_{23}-x_1x_9-x_{10}x_{23}+x_9+x_{23} & -x_{10}x_{23}+x_9-1 \\
1 & 0 & x_{28} & 0 & 0 & x_9x_{28} & -x_{10}x_{28}
\end{bmatrix} .
\endgroup
\]
Since we know $P$ is realizable, the positive part of the slack variety will be nonempty. Taking all variables in the reduced slack matrix $S_F(\xx)$ to be positive, we can obtain a point $\GRVmap(\rho(S_F(\ss))) \in \VV(I_P)_+$ (and hence also restrict our maps to give equivalence of the actual realization spaces) by scaling as follows.
Looking at the monomials in the last row, we can see that the second to last column is already positive, whereas the last column should be multiplied by $-1$.

Next we claim that $\overline{\VV_F} = \overline{\pi_F(\VV(I_P))}$. Here we can prove this directly, without checking the irreducibility condition of Corollary~\ref{COR:closureProjection}. The ideal $K_F$ is generated by the product of the minors in the last two columns of the above matrix, and we can check that $I_F$ is already saturated  by the product of these minors.
Thus we get $ \overline{\overline{\VV_F} \setminus \VV(K_F)} = \overline{\VV_F}$, which proves the claim using Lemma~\ref{LEM:closureProjection}.
\label{EX:cutcube}
\end{example}

\begin{proposition}
Let $P$ be a realizable polytope and $F$ be a set of facets of $P$ satisfying the conditions of Definition~\ref{def:reducedSlackMatrix}. If $\overline{\VV_F}$ is irreducible, then
\[
\VV_F \times \CC^h \cong \VV(I_P)^*
\]
are birationally equivalent, where $h$ denotes the number of facets of $P$ outside $F$.
\label{PROP:reducedSlackModel}
\end{proposition}

\begin{proof} The rational maps are as follows:
\begin{eqnarray*}
\VV_F \times \CC^h & \to & \VV(I_P)^* \\
(\ss_F, \lambda_1,\ldots, \lambda_h) & \mapsto & \GRVmap(\rho(\ss_F))
\end{eqnarray*}
where the $i$th column of $\GRVmap(\rho(\ss_F))$ which is not in $F$ is scaled, so its first nonzero entry is $\lambda_i$, and
\begin{eqnarray*}
\VV(I_P)^* & \to & \VV_F \times \CC^h \\
\ss & \mapsto & (\pi_F(\ss), \gamma_1,\ldots, \gamma_h),
\end{eqnarray*}
where $\gamma_i$ is the first nonzero entry of the $i$th column of $\ss$ which is not in $F$.
That these maps are inverses follows from Lemma~\ref{LEM:slackProjequivalence} and Corollary~\ref{COR:closureProjection}, where $\pi_F(\VV(I_P))$ is nonempty by assumption that $P$ is realizable.
\end{proof}

\begin{example}
Let $P$ be the Perles projectively unique polytope with no rational realization defined in \cite[p. 94]{Grunbaum}. This is an $8$-polytope with $12$ vertices and $34$ facets with the additional feature that it has a non-projectively unique face. Its symbolic slack matrix $S_P(\xx)$ is a $12 \times 34$ matrix with 120 variables.

Let $S_F(\xx)$ be the following submatrix of $S_P(\xx)$ whose 13 columns correspond to nonsimplicial facets:
\[
S_F(\xx) =
\begin{bmatrix}
0 & 0 & 0 & x_{1} & x_{2} & x_{3} & 0 & 0 & 0 & 0 & 0 & 0 & 0\\
0 & 0 & 0 & x_4 & 0 & 0 & x_5 & x_6 & x_7 & 0 & 0 & 0 & 0\\
0 & 0 & 0 & 0 & 0 & 0 & x_8 & 0 & 0 & x_9 & x_{10} & 0 & 0\\
0 & 0 & 0 & 0 & x_{11} & 0 & 0 & 0 & 0 & 0 & 0 & x_{12} & x_{13}\\
0 & 0 & 0 & 0 & 0 & 0 & 0 & x_{14} & 0 & x_{15} & 0 & x_{16} & 0\\
x_{17} & 0 & 0 & 0 & 0 & 0 & 0 & 0 & 0 & 0 & x_{18} & 0 & x_{19}\\
0 & x_{20} & 0 & 0 & 0 & 0 & 0 & 0 & x_{21} & 0 & 0 & 0 & 0\\
0 & 0 & x_{22} & 0 & 0 & x_{23} & 0 & 0 & 0 & 0 & 0 & 0 & 0\\
x_{24} & 0 & 0 & x_{25} & 0 & 0 & 0 & x_{26} & 0 & 0 & 0 & 0 & 0\\
0 & x_{27} & 0 & 0 & x_{28} & 0 & 0 & 0 & 0 & x_{29} & 0 & 0 & 0\\
0 & 0 & x_{30} & 0 & 0 & 0 & x_{31} & 0 & 0 & 0 & 0 & 0 & x_{32}\\
0 & 0 & 0 & 0 & 0 & x_{33} & 0 & 0 & x_{34} & 0 & x_{35} & x_{36} & 0
\end{bmatrix}.
\]
We then scale one entry in each column of $S_F(\xx)$ to $1$; more precisely we set $x_i=1$ for $i=1,3,5,6,7,9,10,11,13,16,17,22,27$. Computing the slack ideal $I_F$ of this scaled matrix, one of the generators of $I_F$ is the polynomial
\[
f(\xx) = x_{15}^2 x_{36}^2 + x_{15} x_{35} x_{36} - x_{35}^2.
\]
Its rehomogenization (with respect to the scaled variables) is
\begin{align*}
H(f(\xx)) & = (x_{10}x_{15}x_{36})^2 + (x_{10}x_{15}x_{36})(x_{9}x_{16}x_{35}) - (x_{9}x_{16}x_{35})^2 \\
	   & = (x_{10}x_{15}x_{36} - \alpha_1x_9x_{16}x_{35})(x_{10}x_{15}x_{36} - \alpha_2x_9x_{16}x_{35}),
\end{align*}
where $\alpha_1 = \frac{-1+\sqrt{5}}{2}, \alpha_2 = \frac{-1-\sqrt{5}}{2}$ are the roots of $x^2+x-1$. In \cite[Theorem 5.3]{GMTWsecondpaper} we show that $\overline{\VV_F}$ decomposes in at least two irreducible components, each coming from a factor of $H(f(\xx))$ and containing a point that maps to $\VV(I_P)^*$, proving that $\VV(I_P)$ is reducible.
\end{example}

\begin{remark}
Using a similar argument we can further reduce the slack matrix. More precisely, once we have the matrix $S_F$, we can remove from it the rows of $S_F$ containing at most $d$ zeros {such that their positions ``span'' the vertex corresponding to that row (i.e., the $k\leq d$ columns containing the zeros form a submatrix of rank~$k$). Furthermore,} the remaining matrix $S_G$ should still contain a flag ($d+1$ triangular submatrix) and have the correct rank $d+1$.
\end{remark}

\begin{example}
Recall the polytope of Example~\ref{EX:cutcube}. Its reduced slack matrix has the form
\[
S_F(\xx) = \left[\begin{array}{ccccccc}
0 & 0 & x_{1} & x_{2} & x_{3} \\
x_{5} & x_{6} & 0 & 0 & x_{7}  \\
0 & x_{9} & 0 & x_{10} & x_{11}  \\
x_{12} & 0 & x_{13} & 0 & x_{14}  \\
0 & 0 & x_{15} & x_{16} & 0  \\
x_{19} & x_{20} & 0 & 0 & 0  \\
0 & x_{23} & 0 & x_{24} & 0  \\
x_{27} & 0 & x_{28} & 0 & 0
\end{array}\right],
\]
and vertices $2,4,5,8$, for example, form a submatrix that still contains a $4\times 4$ triangular submatrix. Thus a ``super reduced'' slack matrix for $P$ would be
\[
\left[\begin{array}{ccccccc}
x_{5} & x_{6} & 0 & 0 & x_{7}  \\
x_{12} & 0 & x_{13} & 0 & x_{14}  \\
0 & 0 & x_{15} & x_{16} & 0  \\
x_{27} & 0 & x_{28} & 0 & 0
\end{array}\right],
\]
which again we scale according to \cite[Lemma 5.2]{GMTWsecondpaper} to get
\[
S_G(\xx)=\left[\begin{array}{ccccccc}
1 & 1 & 0 & 0 & 1  \\
x_{12} & 0 & 1 & 0 & 1  \\
0 & 0 & 1 & 1 & 0  \\
1 & 0 & x_{28} & 0 & 0
\end{array}\right].
\]
Reconstructing rows $6$ and $7$ is a straightforward application of $\GRVmap$ to $S_G(\xx)^\top$. To reconstruct rows $1$ and $3$, notice that we maintain the correct rank and zero pattern by filling row $1$ with $\begin{bmatrix}v_1 & \cdots & v_4\end{bmatrix}S_G(\xx)$ where $\vv^\top = \begin{bmatrix}v_1 & \cdots & v_4\end{bmatrix}$ is an element of the left kernel of the submatrix of $S_G(\xx)$ consisting of columns $1$ and~$2$ (where we want the zeros in row $1$). Since we have $2$ instead of  $d=3$ zeros in these rows, the kernel will be $2$-dimensional instead of $1$-dimensional (i.e. where the column/row is unique up to scaling). Filling out the rows as such we get the following slack matrix.
\[
\begingroup
\setlength\arraycolsep{2pt}
\begin{bmatrix}
0 & 0 & -x_{12}x_{28}y_1\!+\!y_0\!+\!y_1 & y_0 & y_1 \\
1 & 1 & 0 & 0 & 1 \\
0 & x_{12}z_0\!+\!z_1 & 0 & x_{28}z_1\!+\!z_0 & x_{12}z_0\!-\!z_0\!+\!z_1 \\
x_{12} & 0 & 1 & 0 & 1 \\
0 & 0 & 1 & 1 & 0 \\
-x_{12}x_{28}\!+\!x_{28}\!+\!1 & x_{28} & 0 & 0 \\
0 & 1 & 0 & -x_{12}x_{28}\!+\!x_{28}\!+\!1 & 0 \\
1 & 0 & x_{28} & 0 & 0
\end{bmatrix}.
\endgroup
\]
Finally notice that since $S_G(\xx)$ is a $4\times 5$ matrix, its slack ideal is trivial. Thus we see that the dimension of the slack variety from this parametrization is $2$ (free variables $x_{12},x_{28}$) $+\,1+1$ (extra dimension from each $2$-dimensional kernel) $=4$ which agrees with what we found in Example~\ref{EX:cutcube}. To see this explicitly, notice that in the first row, $y_1$ cannot be zero, so we can let $y= y_0/y_1$ and scale the first row by $y_1$. In the third row, $x_{12}z_0\!+\!z_1$ should not be zero, so that we scale this to be one in this row. Then letting $z = \frac{z_0}{x_{12}z_0+z_1}$, we get the following slack matrix with exactly 4 free variables.

\[
\begingroup
\setlength\arraycolsep{5pt}
\begin{bmatrix}
0 & 0 & -x_{12}x_{28}\!+\!y\!+\!1 & y & 1 \\
1 & 1 & 0 & 0 & 1 \\
0 & 1 & 0 & x_{28}-x_{12}x_{28}z+z & 1-z  \\
x_{12} & 0 & 1 & 0 & 1 \\
0 & 0 & 1 & 1 & 0 \\
-x_{12}x_{28}\!+\!x_{28}\!+\!1 & x_{28} & 0 & 0 \\
0 & 1 & 0 & -x_{12}x_{28}\!+\!x_{28}\!+\!1 & 0 \\
1 & 0 & x_{28} & 0 & 0
\end{bmatrix}.
\endgroup
\]
\end{example}

\begin{example}
Recall the nonpolytopal $3$-sphere $P$ of Example \ref{EX:nonrealizableSphere}. Its reduced slack matrix is
\[
S_F(\xx) = \left[\begin{array}{ccccccc}
0 & 0 & 0 & 0 & 0 & x_1\\
0 & 0 & 0 & 0 & x_6 & x_7\\
0 & 0 & x_{10} & x_{11} & x_{12} & 0\\
0 & 0 & x_{14} & x_{15} & 0 & 0\\
0 & x_{18} & 0 & x_{19} & 0 & 0\\
x_{23} & 0 & x_{24} & 0 & 0 & x_{25}\\
x_{27} & x_{28} & 0 & 0 & x_{29} & 0\\
x_{30} & x_{31} & 0 & 0 & 0 & 0
\end{array}\right],
\]
and vertices $1,2,3,5,6,7$ form a submatrix containing a triangle. Thus a ``super reduced'' slack matrix for $P$ with entries scaled according to \cite[Lemma~5.2]{GMTWsecondpaper} would be
\[
S_G(\xx) = \left[\begin{array}{ccccccc}
0 & 0 & 0 & 0 & 0 & 1\\
0 & 0 & 0 & 0 & 1 & 1\\
0 & 0 & 1 & 1 & 1 & 0\\
0 & 1 & 0 & x_{19} & 0 & 0\\
x_{23} & 0 & x_{24} & 0 & 0 & 1\\
1 & 1 & 0 & 0 & 1 & 0
\end{array}\right].
\]
Reconstructing rows $4,8$ by applying $\GRVmap$ to $S_G(\xx)^\top$ and columns $7,8,9,10$ by a second application of $\GRVmap$ to the resulting matrix, we get that $\GRVmap(\rho(S_F(\xx)))$ equals
\[
\begingroup
\setlength\arraycolsep{2pt}
\begin{small}
\begin{bmatrix}
0 & 0 & 0 & 0 & 0 & 1 & 1 & -x_{23} & x_{19}x_{23}^2 & x_{19}^2x_{23}^2x_{24}\\
0 & 0 & 0 & 0 & 1 & 1 & 0 & 0 & x_{19}x_{23}^2 & x_{19}^2x_{23}^2x_{24}\\
0 & 0 & 1 & 1 & 1 & 0 & 0 & 0 & 0 & \colorbox{yellow}{{\bf 0}}\\
0 & 0 & -x_{24} & -x_{19}x_{23} & 0 & 0 & -x_{19}x_{23} & x_{19}x_{23}^2 & 0 & 0\\
0 & 1 & 0 & x_{19} & 0 & 0 & 0 & -x_{19}x_{23}\!-\!x_{23}\!-\!1 & x_{19}x_{23} & x_{19}^2x_{23}x_{24}\\
x_{23} & 0 & x_{24} & 0 & 0 & 1 & x_{19}x_{23}\!+\!x_{23}\!+\!1 & 0 & 0 & 0\\
1 & 1 & 0 & 0 & 1 & 0 & 0 & 0 & 0 & 0\\
x_{19}x_{24} & x_{19}x_{24} & 0 & 0 & 0 & 0 & x_{19}x_{24} & -x_{19}x_{23}x_{24} & \colorbox{yellow}{{\bf 0}} & 0
\end{bmatrix}.
\end{small}
\endgroup
\]
Notice that in the above matrix there are two extra zeros (highlighted) with respect to $S_P(\xx)$. Hence, we can immediately conclude that $P$ is not realizable, without computing the slack ideal $I_G$. Notice that the presence of these extra zeros may depend on the choice of the flag. For some choices we get the extra zeros only after reducing the entries mod $I_G$.
\end{example}

Finally, we illustrate the power of this technique on a quasi-simplicial 4-sphere constructed by Criado and Santos in 2019 \cite{CS19}. This sphere is one of four small (topological) prismatoids constructed as a combinatorial abstraction of the (geometric) prismatoids introduced by Santos in his construction of a counterexample to the Hirsch conjecture. These four prismatoids were shown to be non-Hirsch and shellable, but it is an open question whether they are realizable as polytopes. As a final example of how to apply our new techniques, we show nonrealizability of one such prismatoid below. 

\begin{example} Let $P$ be the abstract polytope, labeled \#1963 in \cite{CS19}, with 14 vertices labeled $0, 1,\ldots, 13$ and with the following 94 facets
\begin{gather*}
\{0,1,2,3,4,5,6\}, \{7,8,9,10,11,12,13\} \\
\{0,1,2,6,10\}, \{0,1,4,7,11\}, \{0,2,3,8,10\}, \{1,2,6,9,13\}, \{1,4,7,8,11\}, \{2,4,7,9,10\}, \\
\{2,5,8,9,10\}, \{5,7,8,10,11\}, \{0,1,2,3,10\}, \{0,2,4,7,11\}, \{1,2,3,8,10\}, \{1,2,5,9,13\}, \\
\{2,4,7,8,11\}, \{1,2,7,9,12\}, \{2,5,8,9,11\}, \{6,7,8,11,12\}, \{0,1,3,4,11\}, \{0,1,3,7,11\}, \\
\{0,2,3,8,11\}, \{1,2,5,9,12\}, \{2,4,7,8,10\}, \{2,5,7,9,12\}, \{2,6,8,9,11\}, \{3,7,8,12,13\}, \\
\{0,2,3,4,11\}, \{0,1,4,7,12\}, \{0,2,6,8,11\}, \{0,1,5,9,12\}, \{2,5,7,8,11\}, \{2,5,7,9,10\}, \\
\{1,3,7,9,13\}, \{5,7,9,10,11\}, \{1,2,3,4,11\}, \{0,2,4,7,12\}, \{1,2,3,8,11\}, \{0,1,3,9,12\}, \\
 \{2,5,7,8,10\}, \{0,5,7,9,12\}, \{1,4,7,9,13\}, \{6,7,9,11,12\}, \{0,1,4,5,12\}, \{0,2,5,7,12\}, \\
 \{1,2,4,8,11\}, \{0,1,5,9,13\}, \{1,3,7,8,11\}, \{0,5,7,9,11\}, \{1,6,8,9,13\}, \{4,7,8,9,10\}, \\
\{0,2,4,5,12\}, \{0,2,5,7,11\}, \{1,2,4,8,10\}, \{0,2,5,9,13\}, \{0,3,7,8,11\}, \{0,6,7,9,11\}, \\
\{0,6,8,9,13\}, \{5,8,9,10,11\}, \{1,2,4,5,12\}, \{1,2,4,7,12\}, \{0,1,6,8,13\}, \{0,2,5,9,11\}, \\
 \{0,6,7,8,11\}, \{0,6,7,9,12\}, \{0,6,8,9,12\}, \{6,8,9,11,12\}, \{0,1,5,6,13\}, \{0,1,3,7,12\}, \\
 \{0,1,3,8,13\}, \{0,2,6,9,13\}, \{0,3,7,8,12\}, \{1,3,7,9,12\}, \{0,3,8,9,12\}, \{3,7,9,12,13\}, \\
 \{0,2,5,6,13\}, \{0,1,6,8,10\}, \{1,2,6,9,10\}, \{0,2,6,9,11\}, \{0,6,7,8,12\}, \{1,6,8,9,10\}, \\
 \{1,4,8,9,13\}, \{4,7,8,9,13\}, \{1,2,5,6,13\}, \{0,2,6,8,10\}, \{1,2,4,9,10\}, \{0,1,3,9,13\}, \\
 \{1,3,7,8,13\}, \{1,4,8,9,10\}, \{0,3,8,9,13\}, \{3,8,9,12,13\}, \{0,1,3,8,10\}, \{1,2,4,7,9\}, \\
 \{1,4,7,8,13\}, \{2,6,8,9,10\}
\end{gather*}

{The set of facets $F = \{F_0, F_1, F_2, F_3, F_4, F_{10}\}$ is a flag and the corresponding} reduced slack matrix for $P$ with the maximum number of variables set to one is the following:
\[
S_F(\xx) = \begin{bmatrix}
0 & 1 & 0 & 0 & 0 & 0 \\
0 & x_{10} & 0 & 0 & 1 & 0 \\
0 & 1 & 0 & x_{18} & 0 & 0 \\
0 & 1 & x_{22} & x_{23} & 0 & 0 \\
0 & x_{33} & x_{34} & 0 & x_{35} & 1 \\
0 & x_{43} & x_{44} & x_{45} & x_{46} & 1 \\
0 & 1 & 0 & 1 & 1 & 1 \\
x_{68} & 0 & x_{69} & 0 & x_{70} & 1 \\
x_{76} & 0 & x_{77} & x_{78} & 0 & 1 \\
x_{85} & 0 & x_{86} & x_{87} & x_{88} & 1 \\
x_{95} & 0 & 0 & 1 & 0 & 0 \\
1 & 0 & 1 & 0 & x_{105} & 1 \\
x_{113} & 0 & x_{114} & x_{115} & x_{116} & 1 \\
x_{127} & 0 & x_{128} & x_{129} & x_{130} & 1
\end{bmatrix}
\]
Now we reconstruct the remaining columns of the slack matrix. We can then recursively determine the sign of each column, as in Example~\ref{EX:cutcube}, by looking for monomial entries and setting the sign of all the entries of that column so they are positive. For example, if we reconstruct facet $\{0,2,6,8,10\}$ we get
$$\begin{bmatrix}
0 \\
-x_{18}x_{77}x_{95}\\
 0 \\
-x_{18}x_{22}x_{95}\\
-x_{18}x_{35}x_{77}x_{95}-x_{18}x_{34}x_{95}+x_{18}x_{77}x_{95}\\
-x_{18}x_{46}x_{77}x_{95}-x_{18}x_{44}x_{95}+x_{18}x_{77}x_{95}\\
 0 \\
-x_{18}x_{70}x_{77}x_{95}-x_{18}x_{69}x_{95}+x_{18}x_{77}x_{95}\\
 0 \\
-x_{18}x_{77}x_{88}x_{95}+x_{18}x_{77}x_{95}-x_{18}x_{86}x_{95}\\
 0 \\
-x_{18}x_{77}x_{95}x_{105}+x_{18}x_{77}x_{95}-x_{18}x_{95}\\
-x_{18}x_{77}x_{95}x_{116}+x_{18}x_{77}x_{95}-x_{18}x_{95}x_{114}\\
-x_{18}x_{77}x_{95}x_{130}+x_{18}x_{77}x_{95}-x_{18}x_{95}x_{128}
\end{bmatrix}
$$
which tells us that we must change the sign of each entry in the column to get positive entries.

From this process, we get a collection of polynomials that must be simultaneously positive. In particular, from this matrix we get polynomials that imply the inequalities $x_{76} > x_{77} > x_{34} > 1$. Furthermore, we have degree 2 polynomials including $-x_{34}x_{76} + x_{34} + x_{76} - x_{77} > 0$. The first inequalities give us
\[
\frac{x_{76}-x_{77}}{x_{76}-1} < 1,
\]
while the second gives
\[
\frac{x_{76} - x_{77}}{x_{76}-1} > x_{34},
\]
which is a contradiction to $x_{34}>1$. Thus we have found a subset of the entries of the slack matrix which cannot be simultaneously positive, so that $P$ is not realizable.
\end{example}

This example illustrates how we have successfully leveraged the connections established earlier in the paper to streamline slack ideal computations. For some polytopes, this simplification can be very significant. As we have seen, we can now automatically tackle, for instance, questions of realizability, rational realizability, or dimensionality of the realization space that were previously, with the standard slack ideal approach, outside the capabilities of computer algebra systems. These examples could potentially be further improved by recourse to numerical tools for analyzing real semialgebraic sets. The new approach to slack computations presented in this section highlights that framing the different models for realization spaces in a common setting, as we do in this paper, not only gives us a useful dictionary to jump between viewpoints, but may allow us to develop new tools for studying realization spaces.

\bibliographystyle{plain}
\bibliography{all}
\end{document}